\documentclass[12pt]{preprint} 

\usepackage[a4paper,innermargin=1.2in,outermargin=1.2in,
bottom=1.5in,marginparwidth=1in,marginparsep
=3mm]{geometry}

\usepackage{fourier} % Use the Adobe Utopia font for the document - comment this line to return to the LaTeX default
\usepackage[english]{babel} % English language/hyphenation
\usepackage{amsmath,amsthm}
\usepackage{amsfonts}
\usepackage{tikz-cd}
\usepackage{mathtools}

\usepackage{shuffle}
\usepackage{ amssymb }
\usepackage{fancyhdr} % Custom headers and footers
\usepackage{comment}

\usepackage{longtable}

\usepackage{cprotect}
\usepackage{hyperref}
\usepackage{mathrsfs}
\usepackage{mhequ}
\usepackage{microtype}
\usepackage{esint}
\usepackage{dirtytalk}

\usepackage{xstring}

\usepackage{orcidlink}

\DeclareMathAlphabet{\mathcal}{OMS}{cmsy}{m}{n}

\colorlet{darkblue}{blue!90!black}
\colorlet{darkred}{red!90!black}
\colorlet{darkgreen}{green!90!black}

\usepackage{array}
\usepackage{tabularx}
\usepackage{authblk}
\usepackage{longtable}
\usepackage{booktabs}
\usepackage{mhequ}

\newcommand{\fraks}{\mathfrak{s}}

\def\gr#1{#1\textnormal{-gr}}
\def\tri#1{#1\textnormal{-tri}}

\newcommand{\NN}{\mathbb{N}}

\newcommand{\mcK}{\mathcal{K}}

\newcommand{\mcM}{\mathcal{M}}

\newcommand{\mfK}{\mathfrak{K}}

\newcommand{\mfd}{\mathfrak{d}}

\renewcommand{\bf}{\mathbf{f}}

\DeclareMathOperator{\EE}{E}

\newcommand{\vertiii}[1]{{\left\vert\kern-0.25ex\left\vert\kern-0.25ex\left\vert #1 
		\right\vert\kern-0.25ex\right\vert\kern-0.25ex\right\vert}}

\newtheorem{theorem}{Theorem}[section]

\newtheorem{corollary}[theorem]{Corollary}
\newtheorem{claim}[theorem]{Claim}
\newtheorem{lemma}[theorem]{Lemma}
\newtheorem{definition}[theorem]{Definition}
\newtheorem{prop}[theorem]{Proposition}

\theoremstyle{remark}
\newtheorem{remark}[theorem]{Remark}

\DeclareMathOperator{\supp}{supp}

\numberwithin{equation}{section} % Number equations within sections (i.e. 1.1, 1.2, 2.1, 2.2 instead of 1, 2, 3, 4)
\numberwithin{figure}{section} % Number figures within sections (i.e. 1.1, 1.2, 2.1, 2.2 instead of 1, 2, 3, 4)
\numberwithin{table}{section} % Number tables within sections (i.e. 1.1, 1.2, 2.1, 2.2 instead of 1, 2, 3, 4)

\colorlet{symbols}{blue!90!black}
\colorlet{testcolor}{green!60!black}

\usetikzlibrary{calc}

\tikzset{
	eps/.style={circle,fill=white,draw=symbols,inner sep=0pt,minimum size=0.8mm},
	}
\makeatletter
\def\DeclareSymbol#1#2#3{\expandafter\gdef\csname MH@symb@#1\endcsname{\tikz[baseline=#2,scale=0.15,draw=symbols]{#3}}}
\def\<#1>{\csname MH@symb@#1\endcsname}
\makeatother

\makeatletter

\DeclareRobustCommand{\TitleEquation}[2]{\texorpdfstring{\StrLeft{\f@series}{1}[\@firstchar]$\if%
		b\@firstchar\boldsymbol{#1}\else#1\fi$}{#2}}

\makeatother

\DeclareMathOperator{\Vol}{Vol}
\DeclareMathOperator{\mass}{\mathbf{m}}

\DeclareMathOperator{\Bd}{Bd}
\DeclareMathOperator{\VV}{V}
\DeclareMathOperator{\diam}{diam}

\newcommand{\osimplex}{\sigma}

\newcommand{\simplex}{\boldsymbol{\sigma}}

\newcommand{\horrule}[1]{\rule{\linewidth}{#1}} % Create horizontal rule command with 1 argument of height

\title{	
\horrule{0.5pt} \\[0.4cm] % Thin top horizontal rule
\large Rough Geometric Integration
 \\ % The assignment title
\horrule{0.5pt} \\[0.5cm] % Thick bottom horizontal rule
}
\author{Ajay Chandra$^{1}$\orcidlink{0000-0003-3690-1890} and Harprit Singh$^2$\orcidlink{0000-0002-9991-8393}}
\institute{
Purdue University, West Lafayette, USA, \email{ajay.chandra@gmail.com} \and
University of Vienna, \email{singhh92@univie.ac.at}}
\begin{document}

\maketitle % Print the title
\begin{abstract}
	\noindent We introduce a notion of distributional $k$-forms on $d$-dimensional manifolds which can be integrated against suitably regular $k$-submanifolds. 
		Our approach combines ideas from Whitney's geometric integration \cite{Whi12} with those of sewing approaches to rough integration \cite{Gub04,FP06}.
\end{abstract}
\\


\tableofcontents
%\newpage

\section{Introduction}
In this work we introduce Banach spaces of generalised/rough differential forms (or cochains) which are analogous to H\"older distributions. 
Following the approach of geometric integration, pioneered by Whitney \cite{Whi12} and further developed by Harrison \cite{GaussGreen, har_notes}, an element $A$ of such a space is characterized by its evaluation on simplices $\sigma \mapsto A(\sigma)$.
We work in a rougher setting where these elements can be of ``infinite variation'', that is we allow $\sup_{\mathcal{K} | \sigma} \sum_{\sigma' \in \mathcal{K}} |A(\sigma')| = \infty$ where the supremum is over subdivisions $\mathcal{K}$ of $\sigma$ into smaller simplices.  
In particular, we can make sense of a class of 
 $k$-forms formally given by $f = \sum_{I} f_{I} \mathrm{d}x^I$  where the $(f_{I})_{I}$ are genuine distributions. 
 
Our work here can be seen as formulating a geometric analogue of the integration theory of Young \cite{Young}, building on earlier work such as \cite{Zus11,ST21,AST24,Che18, Harang2018, CCHS22}. 
It is also motivated by the study of random fields in probability theory where one is often interested in integrating rough random distributions on $\mathbb{R}^{d}$ over $k < d$ dimensional objects, for instance circle averages of the Gaussian Free Field \cite{BP24} or Wilson loop observables for quantum gauge fields \cite{Che18,CCHS22}. 

\subsection{Main results}

We introduce a space $\Omega^{k}_{{\alpha,\beta}}$ of generalised $k$-forms in $d$-dimensional ambient space
for $\alpha \in (0,1]$, $\beta \in [0,1]\cup \{\infty\}$.
Our results for this space are as follows: 
\begin{itemize}
\item Elements of $\Omega^{k}_{{\alpha,\beta}}$ can be integrated in the sense of geometric integration theory over embedded $k$-manifolds.
\item For $\gamma\in (1-\alpha+\beta,1]$,  $\Omega^{k}_{{\alpha,\beta}}$ is a $C^\gamma$--module under pointwise multiplication.
\item Whenever $\alpha>1/2$, these forms can be pulled-back which allows them to be defined intrinsically on manifolds. This allows for a notion of integration by pullback for which a corresponding Stokes' theorem holds. 
\item For $\alpha \in (0,1)$,
there is a natural embedding $\Omega^{k}_{\alpha,\beta}(\mathbb{R}^d) \hookrightarrow \prod_{I}  \mathcal{C}^{\alpha - 1}(\mathbb{R}^{d})$ which extends the mapping
 $f = \sum_{I} f_I(x) dx^{I}\mapsto \{f_I\}_{I}$ on smooth forms $f$.
\item We present a Kolmogorov criterion for the spaces $\Omega^{k}_{{\alpha,\beta}}$, which is checked for a class of  (distributional) Gaussian fields. 
\end{itemize}
Along the way, a useful tool that we formulate and prove is a ``simplicial sewing lemma'', Proposition~\ref{prop:sewing}, which might be of independent interest as it provides a coordinate invariant point of view\footnote{See also Remark~\ref{rem:disclaimer}.} on multidimensional sewing {cf.\ \cite{Harang2018, roughsheet, towghi2002multidimensional}}. 

%
%\blue{The main results about the spaces $\Omega_{\alpha,\beta}$, explained in further detail below are: its elements allow for exterior differentiation; they can be integrated in the sense geometric integration theory à la Whitney; they form a $C^\gamma$ module under pointwise multiplication; there is a natural  embedding $\Omega_{\alpha,\beta}\rightarrow\mathcal{D}'$; they allow for a notion of pull-back and thus are intrinsic to manifolds. }

\subsection{Relationship to past work}
This present work also unifies several previous constructions in geometric integration theory, below we compare our spaces $\Omega_{{\alpha,\beta}}^k$ and results about them to past work in this area.  

\begin{itemize}
\item For $\alpha=\beta=1$ these spaces agree with the flat cochains of Whitney, \cite{Whi12}. For $\alpha=1$, $0<\beta<1$ these spaces agree with the class of $(k-1+\beta)$- cochains introduced in \cite{GaussGreen}, see Remark~\ref{rem:generalised flat norm}.

\item These spaces contain the generalised $d$-forms $f\cdot dg_1\wedge...\wedge dg_d$ introduced in \cite{Zus11}  and recently studied in \cite{AST24} for $d = 2$ by discrete approximation. We extend several results therein, see Remark~\ref{rem:non-atomic}, Corollary~\ref{cor:zust}, and Theorem~\ref{thm:multiplication}.
%, which in particular extends the ``Ito--Stratonivics Theorem'' of \cite{AST24} to $d>2$.
\item 
%Shortly before releasing this work as a preprint, the independent works \cite{B24_Young, B24_Exterior} which also consider a higher dimensional Young type theory, but from a purely geometric measure theory perspective leading to distinct results. 
Shortly before releasing this work as a preprint, the independent works \cite{B24_Young, B24_Exterior} appeared. They also develop a novel higher-dimensional Young-type theory, but from more of a geometric measure theory perspective which gives results distinct from ours.
Moreover, as our article was nearing publication, the preprint \cite{jaffard} appeared on the arXiv, which examines Züst integration.

\end{itemize}

This article was motivated by recent work in probability theory and mathematical physics on Yang--Mills theory.
The problem of defining line integrals of rough $1$-forms naturally arises in this context when one tries to make sense of Wilson loop observables for quantum gauge fields. 
In particular, the works \cite{Che18} and \cite[Sec.~3]{CCHS22} introduce and make crucial use of spaces of rough $1$-forms on $\mathbb{R}^2$
in order to define a well behaved analytic state spaces for two dimensional continuum Yang--Mills theory. 
The spaces introduced therein are very similar to our spaces $\Omega^{1}_{\alpha,\beta}$ for $d=2$ and $k=1$, we refer to
 in Remark~\ref{rem:gauge}, Remark~\ref{rem:control}, Proposition~\ref{prop:embedding}, and Section~\ref{sec:application to random fields} for further discussion.
%As this work was about to go into press, the preprint \cite{jaffard} considered Züst 
%
% Finally, as this work was nearing publication, we learned of the recent preprint \cite{jaffard} which independently constructs an integral generalising Z\"{u}st's integral with a focus on integrating on more irregular domains.  
% 
% where we relate the constructions in the present article to these past works on Yang--Mills. 

%  As our article was nearing publication, the preprint \cite{jaffard} — which examines Züst integration—appeared on the arXiv
\subsection{Outline of the article}
In Section~\ref{sec:preliminaries} we introduce geometric notions, such as the set  $\mathfrak{X}^k$ of oriented $k$-simplices in $\mathbb{R}^d$ and a corresponding space of chains $\mathcal{X}^k=\mathbb{Z}[\mathfrak{X}^k]/{\sim}$ where $\mathbb{Z}[\mathfrak{X}^k]$ denotes the free $\mathbb{Z}$-module generated by $\mathfrak{X}^k$ and the equivalence relation $\sim$ identifies the operators of addition and negation on $\mathbb{Z}[\mathfrak{X}^k]$ with the geometric operations of gluing simplices and inverting orientation.
The subset $\Omega^k$ of measurable elements of the algebraic dual of $\mathcal{X}^k$ provides us with a `universe' of differential forms/cochains.  

The first novelty that appears is our notion of size/mass of a simplex $\mass^k_{\alpha}: \mathfrak{X}^k\to  \mathbb{R}$ for $\alpha\in [0,1]$, given by the maximum of $\Vol^{k-1} (F) \cdot h^{\alpha}_{F}$ over faces $F$ of the simplex, where $h_F$ is the height measured from $F$.

In Section~\ref{sec:distributional} we define spaces $\Omega^k_{{\alpha,\beta}}\subset \Omega^k $ of `distributional' forms by imposing for $ \sigma \in \mathfrak{X}^k$ and $\omega \in  \mathfrak{X}^{k+1}$ the inequalities 
$$| A (\sigma)| \lesssim \mass_\alpha^k (\sigma), \qquad  | \partial A (\omega)| \lesssim \mass_\beta^k (\omega) \ .$$ 
Here $\partial:\mathcal{X}^{k+1}\to \mathcal{X}^k$ is the linear extension of the boundary operator on oriented simplices, which by duality maps $\Omega^k$ to $\Omega^{k+1}$.
Lemma~\ref{lem:equivalent norms} provides equivalent characterisations of the spaces $\Omega^k_{\alpha,\beta}$ by controlling $A(\sigma)$ either in terms of the diameter of $\sigma$, or in terms of the value of $A$ on cubes.
Subsection~\ref{sec:estimates on chains} then introduces a norm $| \cdot |_{(\alpha,\beta)}$ on $\mathcal{X}^{k}$ which for $\alpha=\beta=1$ is equivalent to Whitneys flat norm, see \cite{Whi12}, and for $\alpha=1, \beta\in (0,1)$ equivalent to the $(k-1+\beta)$-norm of Harrison, \cite[Sec.~2]{GaussGreen}. 
%The completion $\mathcal{B}_{\alpha,\beta}^k$ of the normed space $(\mathcal{X}^k, | \cdot |_{(\alpha,\beta)})$
We then establish an integration theory in the spirit of geometric integration by showing that $\Omega^{k}_{\alpha,\beta}$ is contained in the
 continuous dual of $\mathcal{B}_{\alpha,\beta}^k$, the completion of $(\mathcal{X}^k, | \cdot |_{(\alpha,\beta)})$,
which in particular canonically contains closed manifolds.
% In particular, Lemma~\ref{lem:pairing_bound} shows that $\Omega^{k}_{(\alpha,\beta)}$ is contained in the continuous dual  and thus canonically extends to its completion $\mathcal{B}_{\alpha,\beta}^k$, which establishes an integration theory in the spirit of geometric integration. 
%(In Remark~\ref{rem: Manifolds in B}, we observe that closed Manifolds are canonically contained in $\mathcal{B}_{\alpha,\beta}^k$. )

Section~\ref{sec:multiplication} establishes the Young type multiplication theorem, Theorem~\ref{thm:multiplication}, which in particular generalises  \cite[Thm.~4.4]{AST24} from dimension $2$ to arbitrary dimension. 
In Corollary~\ref{cor:zust}, we generalise \cite[Thm.~3.2]{Zus11} and establish that spaces of forms considered there are contained in our spaces. 

The main result of Section~\ref{sec:pull-back} shows that the notion of pull-back on smooth forms can be extended to $\Omega_{\alpha,\beta}^k$. As a consequence one can canonically define spaces $\Omega_{\alpha,\beta}^k(M)$ on any $C^{1,\eta}$-manifold $M$ as soon as $\alpha>1/(1+\eta)$. 
Furthermore, the Stokes' theorem extends to forms in $\Omega_{\alpha,\beta}^k$.

In Section~\ref{sec:embedding},  analogously to \cite[Prop.~3.21]{Che18}, we show that the spaces $\Omega_{\alpha,\beta}^k(\mathbb{R}^{d})$ embed into spaces of Hölder distributions and, for codimension $d - k = 0$, this embedding is an isomorphism.

Section~\ref{sec:application to random fields} states and proves the promised Kolmogorov criterion which is then checked for a class of Gaussian fields.

Lastly, Section~\ref{sec:sewing} contains a proof of the simplicial sewing lemma, Proposition~\ref{prop:sewing}. 
We present the proof, which only requires notions from Section~\ref{sec:preliminaries}, at the end of the article in order to streamline exposition. For the proof we define the notion of a ``strongly regular method of subdivision'' (see Definition~\ref{def:reg_method}) -- the existence of such a method is non-trivial and follows from \cite{EG}.

\paragraph{Notation:}
We shall often use for the notation $\lesssim$ to mean that an inequality holds up to multiplication by a constant which may change from line to line but is uniform over any stated quantities. We furthermore write $\lesssim_{C}$ to emphasise the dependence of that constant on $C$. We write 
$a\vee b:=\max\{a,b\}$ and $a\wedge b:=\min\{a,b\}$ 
for $a,b\in \mathbb{R}$.

For $m\in \mathbb{N}$ we equip $\mathbb{R}^m$ with the Borel $\sigma$-algebra and a norm $|\cdot|$ which is fixed throughout the article. We write $\Vol^k$ for the $k$-dimensional Hausdorff measure. 
For a function $f: \mathbb{R}^d \to \mathbb{R}^m$, we shall use the notation $Df=\{\partial_{i} f\}_{i=1}^d$ to denote its derivative, as well as multi-index notation $D^I f:= \partial_1^{i_1},...,\partial_d^{i_d} f$ for $I=(i_1,...,i_d)\in \mathbb{N}^{d}$. We extend differentiation to distributions in the usual way. For $\Omega \subset \mathbb{R}^d$ and a function $f:\Omega \to \mathbb{R}^m$ we write $\|f\|_{L^{\infty}(\Omega)}:= \sup_{x\in \Omega} |f|$.  For $\alpha\in (0,1]$, we denote by $C^{\alpha}(\Omega)$ the usual space of Hölder continuous functions equipped with the norm 
$$\|f \|_{C^\alpha(\Omega)}:= \|f\|_{L^{\infty}(\Omega)} + \sup_{x,y\in \Omega} \frac{|f(x)-f(y)| }{|x-y|^\alpha\wedge 1}\;.$$
As standard, $C^{k,\alpha}(\Omega)$ denotes $k$-times differentiable function such that the $k$- derivatives is are elements of $C^{\alpha}$. For $\alpha<0$, we work with the standard distribution spaces $C^\alpha$ as in \cite{Hai14}.
Finally, we use both $d$ and $\partial$ to denote exterior derivatives. 

\paragraph{Acknowledgements:}
The authors would like to thank Philippe Bouafia for bringing his preprints to their attention and for interesting discussions comparing our methods. 
They also thank the anonymous referees for many useful comments which improved this paper, as well as Thomas Jaffard for spotting a typo in \eqref{eq:comparison1}.
HS would like to thank M\'at\'e Gerencs\'er for his hospitality during a stay at TU Wien and gratefully acknowledges financial support from the EPSRC via Ilya Chevyrev’s New Investigator Award EP/X015688/1.
AC gratefully acknowledges partial support by the EPSRC through the “Mathematics of Random Systems” CDT EP/S023925/1.

\section{Preliminary geometric notions}\label{sec:preliminaries}
	
In this section we recall standard notation and facts about simplices and chains, cf. \cite{Hat05, Mun18, Lee10, GaussGreen}.
Unless mentioned otherwise, we shall be working in the ambient space $\mathbb{R}^d$.

\subsection{Simplices and chains}
Vectors $v_0,v_1,...v_k\in \mathbb{R}^d$ are called \textit{affinely independent} if they lie in a unique $k$-dimensional affine subspace of $\mathbb{R}^d$, or equivalently, if the vectors {$\{v_i-v_0\}_{i=1}^k$} are linearly independent.
	
\begin{definition}\label{def:non-oriented}
	The (unoriented) $k$-simplex $\simplex$ spanned by affinely independent vectors $v_0,...,v_k\in \mathbb{R}^d$ is defined as
	$$\simplex:=\llbracket v_0,..,v_k \rrbracket:=\left\{ \sum_{i=0}^k t_i v_i\in \mathbb{R}^d \ : \ t_i\in [0,1], \ \sum_{i=0}^k t_i=1 \right\} \ .$$

The points $v_i$ are called the vertices of $\simplex$ and we write $\VV(\simplex)$ for this set.
A simplex spanned by a non-empty subset of \ $\VV(\simplex)$ is called a face of $\simplex$. 
A face which is a $k-1$-simplex is called a boundary face. 
\end{definition}

%
%\begin{remark}\label{rem:volume...}
%Then, 
% $\Vol_{k}(\sigma)\leq d(v,w) \left( \Vol^{k-1}(\sigma\setminus v) \wedge  \Vol^{k-1}(\sigma\setminus w) \right)$
% for any distinct $v,w\in \VV$.
%\end{remark} \harprit{check if consistently used, and reference}

\begin{definition}\label{def:subdivision,mesh}
We say a collection $\simplex_1,\dots,\simplex_N$ of $k$-simplices are non-overlapping if $\Vol^k(\simplex_i \cap \simplex_j) = 0$ for every $1 \le i < j \le N$.

A \textit{subdivision} of a simplex $\pmb{\sigma}$ consists of a  finite set of non-overlapping k-simplices $\mathcal{K}$ 
such that $\pmb{\sigma}=\bigcup_{\pmb{w}\in \mathcal{K}} \pmb{w}$. 

The \textit{mesh} of the subdivision $\mathcal{K}$ is defined as $\diam(\mathcal{K}):=\max_{\pmb{w} \in \mathcal{K} } \diam (\pmb{w})$. 
We call a sequence of subdivisions $(\mcK_n)_{n\in \mathbb{N}}$ of a $k$-simplex $\simplex$ regular if, for any $\beta>k$, 
$$\lim_{n\to \infty} \sum_{\pmb{w}\in \mathcal{K}_n}\diam(\pmb{w})^\beta \to 0\ .$$
\end{definition}

An \textit{orientation} of a simplex spanned by affinely independent vectors $v_0,..,v_k\in \mathbb{R}^d$ is an ordering of its vertices $(v_0,...v_k)$, where two orderings are declared equivalent if they differ by an even permutation. 

\begin{definition}\label{def:simplex, etc}
An oriented simplex $\osimplex$ consists of an unoriented simplex $\simplex$ together with an orientation $o$ of its vertices, that is 
$	\osimplex=(\simplex, o)$.
We shall use the notation $[v_0,..,v_k]$ to denote the simplex $\llbracket v_0,..,v_k \rrbracket$ with the orientation inherited from the ordering $(v_0,..,v_k)$.
\end{definition}

We denote by $\mathfrak{X}^k$ the set of all oriented $k$-simplices. 
Given $\mathfrak{K} \subset \mathbb{R}^{d}$, we define $\mathfrak{X}^{k}(\mathfrak{K})$ to be the set of oriented simplices $\osimplex \in \mathfrak{X}^{k}$ with $\simplex \subset \mathfrak{K}$.
For $R > 0$, we also write $\mathfrak{X}_{\le R}^{k}(\mathfrak{K})$ for the set of all $\osimplex \in \mathfrak{X}^{k}(\mathfrak{K})$ with $\diam(\osimplex) \leq R$.

Note that $0$-simplices only have one orientation. For $k \geq 1$ a $k$-simplex $\simplex$ can carry one of two distinct orientations and for a given orientation $o$, we write $-o$ to denote the other orientation on $\simplex$. We observe that the orientation $o$ of a simplex $\simplex$ can be canonically identified with an orientation of the (unique) $k$-dimensional hyperplane containing $\simplex$ -- we shall freely use this identification
and compare the orientation of two simplices contained in the same hyperplane.
We also extend the notion of non-overlapping to oriented simplices in the natural way. 

\begin{definition}
For a simplex $\osimplex=[v_0,...,v_k]$, we denote by $\sigma_{\setminus v_i}:= (-1)^i[v_0, ... ,\hat{v}_i, ... ,v_k]$ the (oriented) face which does not contain $v_i\in V(\sigma)$, where we have used standard notation which means leaving out the term with a hat. 
We set $\Bd(\sigma):= \{ \sigma_{\setminus v} \ : \ v\in V(\sigma)\}$.
\end{definition}

Given an oriented simplex $\osimplex$, we write $\simplex$ for the underlying non-oriented simplex. 
We sometimes overload notation by allowing functions which take unoriented simplices $\osimplex$ as arguments to also take oriented simplices $\osimplex = (\simplex,o)$ as arguments by just disregarding the orientation $o$.

\begin{definition}\label{def:subdivision_2}
Given an oriented simplex $\osimplex \in \mathfrak{X}^{k}$, we say $\mathcal{K} \subset \mathfrak{X}^{k}$ is a subdivision of $\osimplex$ if every $\tau \in \mathcal{K}$ has the same orientation as $\osimplex$ and $\big\{ \pmb{\tau}: \tau \in \mathcal{K} \big\}$ is a subdivision of the non-oriented simplex $\pmb{\sigma}$ (see Definition~\ref{def:subdivision,mesh}).

We write $\mathcal{K} | \sigma$ to denote that $\mathcal{K}$ is a subdivision of the oriented simplex $\sigma$.
Note that the notions of mesh and of regularity clearly extend to the subdivisions of oriented simplicies.  
\end{definition}

\begin{definition}\label{def:chain}
The set of $k$-chains is given by $\mathcal{X}^k= \mathbb{Z}(\mathfrak{X}^k)/ \sim$ where the equivalence relation $\sim$ for 
$k=0$ is trivial and equivalence classes consist of singletons. For $k\geq 1$ 
\begin{itemize}
\item $-\osimplex \sim (\simplex,-o) \ $ for every $\osimplex = (\simplex,o) \in \mathfrak{X}^k$.

\item $
\displaystyle\sum_{\tau \in \mathcal{K}}  \tau \sim \osimplex\;
$ for any $\osimplex \in \mathfrak{X}^{k}$ and subdivision $\mathcal{K} | \sigma$.
\end{itemize}
\end{definition}
%\gray{\begin{remark}\label{rem:adding lower dimensional stuff}
%In our analysis, we will sometimes write expressions that could correspond to degenerate $k$-simplices, that is expressions $[v_{0},\dots,v_{k}]$ where the vertices $(v_j)_{j=0}^{k}$ are not affinely independent. 
%In this case, we identify $[v_{0},\dots,v_{k}]$ with $0$ in  $\mathfrak{X}^k$ 
%\end{remark}
%}

We recall the boundary operator 
\begin{equ}\label{eq:boundary op}
\partial : \mathfrak{X}^k \to \mathcal{X}^{k-1}, \qquad \sigma\mapsto \partial \sigma:= \sum_{v\in \VV(\sigma)} \sigma_{\setminus v} = \sum_{F\in \Bd(\sigma)} F \;.
\end{equ}
Note that this induces a map $\partial: \mathcal{X}^{k}\to \mathcal{X}^{k-1}$
which satisfies
$\partial\circ \partial=0$, cf.\ \cite{Hat05}.

\begin{definition}\label{def:cube}
We call $\pmb{Q} \subset \mathbb{R}^{d}$ a (non-oriented) $k$-cube of side length $r>0$ if it is the image of $[0,r]^k\subset \mathbb{R}^k$ under an isometric embedding. 
An oriented cube is specified by $Q= (\pmb{Q}, o)$, where $o$ is an orientation of the $k$-hyperplane containing $Q$. 
For a set $\mathfrak{K}\subset \mathbb{R}^d$, $R>0$, we write  $\mathfrak{Q}^k_{\leq R}(\mathfrak{K})$ for the set of oriented cubes $Q= (\pmb{Q}, o)$, where $\pmb{Q} \subset \mathfrak{K}$ has side--length bounded by $R$.

We view $Q = (\pmb{Q},o)$ as an element of $\mathcal{X}^{k}$ by setting $Q = \sum_{\pmb{\tau} \in \mathcal{K}} (\pmb{\tau},o)$ for any $\mathcal{K} \subset \mathfrak{X}^{k}$ consisting of non-overlapping simplices which satisfy $\bigcup_{\pmb{\tau} \in \mathcal{K}} \pmb{\tau} = \pmb{Q}$.

Note that the notion of non-overlapping clearly extends to cubes.
\end{definition}

Working with cubes or simplices both offer distinct advantages. 
While cubes are advantageous when working on fibered spaces, cf. \cite{Harang2018},
simplices are combinatorially simpler. 
For example the boundary faces of a simplex can easily by explicitly expressed and any polygon can be decomposed into finitely many simplices. 
We recall here a special case of Whitney's covering theorem \cite{Whitney_extension}, which will allow us to interchange between control over simplices and control over cubes, see Lemma~\ref{lem:equivalent norms}.
\begin{lemma}\label{lem:covering}
Given any non-oriented $k$-simplex $\simplex$, there exists a family of non-overlapping $k$-cubes $\big( \pmb{Q}^{i}_{2^{-n}} : n \in \mathbb{N}, i \in I_{n} \big)$ where each $I_{n}$ is a finite set, $\pmb{Q}^{i}_{2^{-n}}$ is $k$-cube of side length $2^{-n}$, and 
\[
\simplex 
=
\overline{ \bigcup_{n=0}^{\infty} \bigcup_{i \in I_{n}} \pmb{Q}^{i}_{2^{-n}}} 
\quad
\text{and}
\quad
2^{-n}\sqrt{k}
\le
\mathrm{dist}(\pmb{Q}^i_{2^{-n}}, \mathbb{R}^{d} \setminus \simplex) 
\le 4 \cdot 2^{-n} \sqrt{k}\;.
\]
%A covering as in the above lemma is called a Whitney decomposition. If $\diam(\simplex)\leq 1$ one checks that 
%\begin{equ}\label{eq:bound on I_n}
%|I_n|\lesssim_k 2^{n(k-1)}\;.
%\end{equ}
Moreover, the Whitney decomposition can be chosen so that for any $\simplex$, we can find index sets 
$(J_{N} : N \in \mathbb{N})$ and collections of non-overlapping simplices $\{\simplex_{j} : j \in J_{N}\} $ of diameter at most $2^{-N}$ 
%$\subset \displaystyle\mathfrak{X}^{k}_{\le 2^{-N}}$ 
such that
\begin{equ}\label{whitney_remainder}
\simplex - \sum_{n = 0}^{N} \sum_{ i\in I_n} \pmb{Q}^{i}_{2^{-n}} 
=\sum_{j \in J_{N}}  \simplex_j
\quad
\text{and}
\quad  
|J_N|
% \sum_{n = 0}^{N} |I_n| 
\lesssim_k  2^{N(k-1)}\;,
\end{equ}
where the first equality above is between elements in $\mathcal{X}^{k}$ and the second inequality is uniform in $N$. 
\end{lemma} 
\begin{remark}
We extend the notion of Whitney decomposition to oriented simplices $\simplex$ in the obvious way, that is, one has a family of oriented cubes $\big( Q^{i}_{2^{-n}} : n \in \mathbb{N}, i \in I_{n} \big)$ where the $\pmb{Q}^{i}_{2^{-n}}$ form a Whitney decomposition of $\pmb{\sigma}$, and the $Q^{i}_{2^{-n}}$ are oriented as $\sigma$ is. 
\end{remark}

\begin{remark}\label{rem: partition of unity}
Lastly, we explain how to construct a smooth partition of unity subordinate to the Whitney decomposition of a $d$-simplex $\simplex\subset \mathbb{R}^d$.
%We write $\pmb{Q}_0=[-1/2,1/2]^d\subset \mathbb{R}^d$ and 
 For an axis parallel cube $\pmb{Q}\subset \mathbb{R}^d$ we denote by 
$x_{\pmb{Q}}\in \mathbb{R}^d$ resp. $r_{\pmb{Q}}\in \mathbb{R}^d$ the center, respectively the side length of $\pmb{Q}$ which are characterised by 
$\pmb{Q}= x_{\pmb{Q}}+ r_{\pmb{Q}} \cdot [-1/2,1/2]^d$.
Let $\phi\in C^\infty$ be compactly supported on $[-2/3,2/3]^d$ such that $\phi|_{[-1/2,1/2]^d}=1$, 
we define
$$\phi_{\pmb{Q}}(x)= \phi \left(\frac{x-x_{\pmb{Q}}}{r_{\pmb{Q}}} \right)\;.$$ 
For a given Whitney decomposition  $\simplex 
=
\overline{ \bigcup_{n=0}^{\infty} \bigcup_{i \in I_{n}} \pmb{Q}^{i}_{2^{-n}}} $ set
 $$\phi_{n,i}(x) := \begin{cases}\frac{  
 \phi_{\pmb{Q}^{i}_{2^{-n}}}(x)
 }{\displaystyle\sum_{n\in \mathbb{N},i\in I_n}  \phi_{\pmb{Q}^{i}_{2^{-n}} } (x)} & \text{if } x\in \text{Int}(\simplex)\\
 0 & \text{else.}
\end{cases}
$$
We call $\{\phi_{n,i}\}_{n\in \mathbb{N}, i\in I_n}$ the partition of unity subordinate to the above Whitney decomposition constructed from $\phi\in C^\infty$. Finally, one notes that by construction
$$ \|D^k \phi_{n,i}\|_{L^\infty} \lesssim_k 2^{nk}$$
uniformly over $\simplex$, and $n\in \mathbb{N}, i\in I_n$.
\end{remark}

%\ajay{Have put the stuff about remainders of Whitney extensions in the lemma, so we're precise about we are assuming, have added argument you gave in gray as a commented out proof below. }
%\begin{proof}
%Indeed, by construction of the Whitney decomposition $\sigma$ is contained 
%in the union of $K_N\sim \frac{\Vol^{k}(\sigma)}{2^{-kN}}$ many axis parallel cubes which either contained in or are non-overlapping with the cubes $\big\{\pmb{Q}^{i}_{2^{-n}}\big\}_{n\leq N, i\in I_n}$. Since all remaining points of $\sigma$ are contained cubes intersecting $\mfd_{N}\sigma$, there are at most 
%$ 2^{N(k-1)}$ many.
%\end{proof}

\subsection{Controls on simplices}

For any non-oriented $j$-simplex $\mathbf{F}$, we denote by $\langle \pmb{F} \rangle \subset \mathbb{R}^{d}$ the $j$-hyperplane containing $\pmb{F}$. 
We define for $\alpha\in (0,1]$ and any $k$-simplex $\pmb{\sigma}$
\begin{equ}\label{def:mass}
h_{\pmb{\sigma}}:=\min_{ \pmb{F} \in \Bd(\pmb{\sigma})} \sup\{d(x, \langle \pmb{F} \rangle ) \ : \ x\in \pmb{\sigma} \}
\quad
\text{and}
\quad
{\mass}_\alpha^{k} (\pmb{\sigma}) := \max_{ \pmb{F} \in \Bd(\pmb\sigma)} \Vol^{k-1}( \pmb{F}) h_{\pmb{\sigma}}^{\alpha}\;.
\end{equ}
We observe that, 
\begin{equ}\label{eq:comparison0}
h_{\pmb{\sigma}}^{1-\alpha} {\mass}_\alpha^{k} (\pmb{\sigma}) = \frac{1}{k} \Vol^{k}(\pmb{\sigma})\;,
\end{equ}
and, whenever $\diam(\pmb{\sigma})\leq 1$,
%\begin{equ}\label{eq:comparison1}
%{\mass}^{k}_\alpha(\pmb{\sigma})\leq \Big( \frac{1}{(k-1)!}\diam (\pmb{\sigma})^{k-1+\alpha}  \Big) \wedge \Big( k^{\alpha} \Vol^k{(\pmb{\sigma})}^\alpha\Big) \ . 
%\end{equ} 
\begin{equ}\label{eq:comparison1}
{\mass}^{k}_\alpha(\pmb{\sigma})\leq \diam (\pmb{\sigma})^{k-1+\alpha}   \wedge \big(k\Vol^k{(\pmb{\sigma})}\big)^\alpha \ ,
\end{equ} 
where the first upper bound follows from \eqref{def:mass} and the second from \eqref{eq:comparison0} and \eqref{eq:comparison1}.\footnote{
For the first note that $\Vol^{k-1}( \pmb{F})\leq \diam (\pmb{\sigma})^{k-1}$ and $h_{\pmb{\sigma}}\leq\diam (\pmb{\sigma})$, and for the second that $h_{\pmb{\sigma}}^{\alpha} \max_{ \pmb{F} \in \Bd(\pmb\sigma)} \Vol^{k-1}( \pmb{F}) \leq h_{\sigma}^{\alpha} \Big(\max_{ \pmb{F} \in \Bd(\pmb\sigma)} \Vol^{k-1}( \pmb{F})\Big)^\alpha = \big(k \Vol^{k}(\pmb{\sigma}) \big)^{\alpha}$ which holds since $\diam(\pmb{\sigma})\leq 1$.}

We also set $\mass_\infty^k (\pmb{\sigma})=0$ and $\mass_0^k(\pmb{\sigma})=1$.
\begin{remark}
We observe that the left-hand side of the inequality \eqref{eq:comparison1} is close to sharp when either the simplex $\pmb{\sigma}$ is equilateral or when $\pmb{\sigma}$ is a cone of small height over an equilateral simplex. 
%The second inequality of \eqref{eq:comparison} is sharp resp. not sharp in the exact opposite situations.
\end{remark}

\begin{lemma}\label{lem:(alpha,k) mass}
Suppose we have $k$-simplices $\simplex, \simplex_1,\dots,\simplex_{n}$ such that $\diam (\simplex) \leq 1$ and $\simplex= \bigcup_{i=1}^{n} \simplex_i$, then
$${\mass}^k_\alpha(\pmb{\sigma}) \leq \sum_{i=1}^{n} {\mass}^k_\alpha(\pmb{\sigma}_i)\;.$$
\end{lemma}
\begin{proof}
Using \eqref{eq:comparison0} we have
\[
{\mass}^k_\alpha(\pmb{\sigma}) = \frac{\Vol^k(\pmb{\sigma})}{k h^{1-\alpha}_{\pmb{\sigma}}}
\leq \sum_{i=1}^{n}  \frac{\Vol^k(\pmb{\sigma}_i)}{k h^{1-\alpha}_{\pmb{\sigma} }}
= \sum_{i=1}^{n}  {\mass}^k_\alpha(\pmb{\sigma}_i)   \frac{k h^{1-\alpha}_{\pmb{\sigma}_i}}{k h^{1-\alpha}_{\pmb{\sigma}}} 
\leq \sum_{i=1}^{n}  {\mass}^k_\alpha(\pmb{\sigma}_i)  
\]
where we used that
$h^{1-\alpha}_{\sigma_{i}} \leq h^{1-\alpha}_{\sigma}$, and that $\Vol^k(\cdot)$ is additive.
%
%
%and that 
%$h^{1-\alpha}_{\sigma_i}\leq h^{1-\alpha}_{\sigma_i}$ \ajay{Think you mean $h^{1-\alpha}_{\sigma_{i}} \leq h^{1-\alpha}_{\sigma}$ here right?} . The claim follows from the additivity of $\Vol^k$
\end{proof}
\begin{lemma}\label{lem:diam_mass characterisation}
There exists a constant $C=C(k)$, such that$${\mass}^k_\alpha(\pmb{\sigma}) \leq \inf 
\Big\{  \sum_{j \in J}  \diam (\simplex_j)^{k-1+\alpha}: 
\pmb{\sigma} = \bigcup_{j \in J} \simplex_j  \Big\}  \leq C  {\mass}^k_\alpha(\pmb{\sigma}) \ .$$
for all $k$-simplices $\simplex$ with $\diam(\pmb{\sigma})\leq 1$ and $\alpha\in (0,1]$.
\end{lemma}
\begin{proof}
The first inequality above follows by combining \eqref{eq:comparison1} and Lemma~\ref{lem:(alpha,k) mass}. 
For the second inequality above, observe that one can write
$\simplex= \bigcup_{i\in I} \simplex_i$ where $\{ \simplex_{i} \}_{i \in I}$ are non-overlapping, $\diam (\simplex_i) \leq h_{\simplex}$, and $|I| \lesssim_{k} \frac{ \Vol^{k-1}(\pmb{F})}{h_{\simplex}^{k-1}}$ -- here $\pmb{F} \in \Bd(\simplex)$ is where the maximum in the second definition of \eqref{def:mass} is achieved.\footnote{Let $v$ be the vertex of $\simplex$ such that $\pmb{F} = \simplex_{/v}$, such a subdivision of $\simplex$ can then be generated by appropriately subdividing $\pmb{F}$ into $\{\pmb{F}\}_{i}$ and taking the vertices of $\pmb{\sigma}_{i}$ to be given by the vertices of $\mathbf{F}_i$ along with $v$, and if necessary iterating this procedure.}

Using this subdivision, we have  
$$
\sum_{i \in I}  \diam (\simplex_i)^{k-1+\alpha} \lesssim_{k} \Vol^{k-1}(\pmb{F}) h_{\simplex}^{\alpha} = \mass_\alpha (\simplex)\;.
$$
\end{proof}

%
%Given $M \in \mathbb{N}$, $\alpha, \beta\in (0,1]$ we set for $X \in \mathcal{X}^{k}(\mathfrak{K})$ 
%\begin{equ}\label{eq:flat_norm}
%|X|_{(\alpha,\beta),M;\mathfrak{K}}
%= 
%\inf
%\Bigg\{
%| X- \partial^{k+1} Z |_{\alpha;k}
%+ 
%| Z |_{\beta;{k+1}}
%: 
%Z \in \mathcal{Z}(k+1,M,\mathfrak{K})
%\Bigg\}\;.
%\end{equ}
%
%\begin{remark}
%Note that this construction works for any sub additive control instead of $\mass_\alpha$.
%\end{remark}

\subsection{Functions on chains}
\begin{definition}\label{def:kform}
We denote by $\Omega^k\subset(\mathcal{X}^k)^* $ the subspace of the
algebraic dual of $\mathcal{X}^k$ consisting of all elements $A\in (\mathcal{X}^k)^*$ 
such that the map $(\mathbb{R}^d)^{k+1 }\ni (v_0,\dots,v_k) \mapsto A[v_0,...,v_k] \in \mathbb{R}$
is Borel-measurable.
\end{definition} 
\begin{remark}\label{rem:relation}
Observe that any $A \in \Omega^{k}$ can canonically be identified with
a (measurable) map $A: \mathfrak{X}^k\to \mathbb{R}$ satisfying the following two properties 
\begin{itemize}
\item If $k \ge 1$, then for any oriented $k$-simplex $\osimplex$ one has $A(-\osimplex)= -A(\osimplex)$.
\item Given $\osimplex \in \mathfrak{X}^{k}$ and $\mathcal{K} | \osimplex$,  
\[
A(\osimplex) = \sum_{\tau \in \mathcal{K}} A(\tau)\;. 
\]
\end{itemize}
%We denote by $\Omega^k$ the space of all such additive maps on $\mathfrak{X}^{k}$. 
%
% $A:\mathcal{X}^{k} \rightarrow \mathbb{R}$. 
%Moreover, in the context of Remark~\ref{rem:adding lower dimensional stuff}, we extend $A$ to degenerate simplices $\sigma$ by setting $A(\sigma)=0$. 
\end{remark}

Note that there is a natural corresponding space of $k$-chains $\mathcal{X}^{k}(\mathfrak{K})$ and additive maps $\Omega^{k}(\mathfrak{K})$.

\begin{lemma}\label{lem:equivalent norms}
Let $\alpha \in (0,1]$. Then, for any $A \in \Omega^k$, the following three conditions are equivalent: 
\begin{enumerate}
\item\label{item_1}
$|A(Q)|\lesssim \diam(Q)^{k-1 + \alpha} $ uniformly over  $Q\in \mathfrak{Q}_{\leq 1}^k$ and, for any Whitney decomposition $\big( {Q}^{i}_{2^{-n}} : n \in \mathbb{N}, i\in I_{n} \big)$ of $\sigma$, $\sum_{n\in \mathbb{N}} \sum_{ i\in I_n} A({Q}^{i}_{2^{-n}})$ converges absolutely to $A(\sigma)$. 
%, then $\sigma\in \mathfrak{X}_{\leq 1}^{k}$ 
%$$|A(\sigma)|\lesssim \mass_{\eta-(k-1)}^{k}(\sigma)\ ,$$
%uniformly over $\sigma\in \mathfrak{X}_{\leq 1}^{k}$. 
\item \label{item_2}
$|A(\sigma)|\lesssim \mass_{\alpha}^{k}(\sigma) \ , $ uniformly over $\sigma\in \mathfrak{X}_{\leq 1}^{k}$. 
\item\label{item_3} $|A(\sigma)|\lesssim \diam(\sigma)^{k-1 + \alpha}$, uniformly over $\sigma\in \mathfrak{X}_{\leq 1}^{k}$. 
\end{enumerate}
\end{lemma}
\begin{remark}
Note that the additional assumption in Item~\ref{item_1} on additivity over Whitney decompositions is needed due to the fact that one can in general not write a simplex as a union of finitely many cubes. 
\end{remark}
\begin{proof}
Throughout this proof we write $\eta = k - 1 + \alpha$. 
We first prove Item~\ref{item_1} implies Item~\ref{item_2}. 
We apply Lemma~\ref{lem:covering} and write $( Q^{i}_{2^{-n}} : n \in \mathbb{N}, i \in I_{n})$ for a Whitney decomposition of $\simplex$. 

We define
%$\mfd_{2^{-n}; \sigma}$ for the set of points $p\in \pmb{\sigma}$ such that $\ 2^{-n-1}\sqrt{k}
%\le
%\mathrm{dist}(p, \mathbb{R}^{d} \setminus \simplex) 
%\le 4 \cdot 2^{-n+1} \sqrt{k}$
\begin{equ}
\mfd_{n} \pmb{\sigma}
:= 
\big\{ p\in \pmb{\sigma} \ : \ 2^{-n-1}\sqrt{k}
\le
\mathrm{dist} \big(p, \mathbb{R}^{d} \setminus \simplex \big) 
\le 4 \cdot 2^{-n+1} \sqrt{k} \big\}\ ,
\end{equ}
Next we note that  $2^{-n}>h_{\pmb{\sigma}} \Rightarrow  I_n = \emptyset$ and for uniform in $n \in \mathbb{N}$,  $|I_n|\lesssim_{k}  \Vol^{k}( \mfd_{n} \pmb{\sigma} ) 2^{nk}$.
It follows that
%Thus, using \eqref{eqlocal:distance bound} in the third inequality and writing $ \mathfrak{d}_{2^{-n}; \sigma}:= \{ p\in \pmb{\sigma} \ : \ d(p, \partial\sigma)\sim 2^{-n} \}$ we conclude

\begin{equs}\label{eqs:estimates_whitney}
|A(\sigma)| &\leq \sum_{n = 0}^{\infty} \sum_{i\in I_n} |A(Q_{2^{-n}}^{i})| 
 \lesssim \sum_{n=0}^{\infty} |I_n| 2^{-n\eta}
\lesssim \sum_{n \ge -\log_{2}(h_{\pmb{\sigma}})}  \Vol^{k}(\mfd_{n} \pmb{\sigma})    2^{-n(\eta - k)}\\
&\lesssim \Vol^{k}(\pmb{\sigma}) \max_{ n \ge -\log_{2}(h_{\pmb{\sigma}})}2^{-n(\eta -k)}
\leq \Vol^{k}(\pmb{\sigma})  h_{\pmb{\sigma}}^{\eta-k}
=
\frac{\Vol^{k}(\pmb{\sigma})}{h_{\pmb{\sigma}}  h_{\pmb{\sigma}}^{\eta-{k-1}}} 
\leq \mass_{\eta-(k-1)}^{k}(\pmb{\sigma})\;,
\end{equs}
where in the first inequality on the second line above we used that $\sum_{n = 0}^{\infty} \Vol^{k}(\mfd_{n} \pmb{\sigma}) \leq 4 \cdot \Vol^{k}(\pmb{\sigma})$. 
This concludes showing Item~\ref{item_1} implies Item~\ref{item_2}. 

The fact that Item~\ref{item_2} implies Item~\ref{item_3} follows from the last inequality in \eqref{eq:comparison1}.

Finally, to argue that Item~\ref{item_1}  follows from Item~\ref{item_3} we first observe that we can realize any $Q\in \mathfrak{Q}_{\leq 1}^k$ of diameter $r$  as a sum of $k!$ oriented simplices of diameter at most $r$ and thus the estimate on $A(Q)$ in Item~\ref{item_1} follows directly. 

For the rest of  Item~\ref{item_1}, given a Whitney decomposition of $\sigma$ we note that the absolute convergence of $\sum_{n \in \mathbb{N}} \sum_{i \in I_{n}} A(Q^{i}_{2^{-n}})$ follows from the estimate from the first part of Item~\ref{item_1} and proceeding as in \eqref{eqs:estimates_whitney}.
To show convergence to $A(\sigma)$, using the notation of \eqref{whitney_remainder} we have
\[
\Big|A(\sigma) - \sum_{n = 0}^{N} \sum_{ i\in I_n}   A({Q}^{i}_{2^{-n}}) \Big|\leq \sum_{j\in J_N}  |A(\sigma_j )| \lesssim 
2^{ -N(k-1 + \alpha)} 2^{N(k-1)} \xrightarrow[N \rightarrow \infty]{} 0\;.
\]
\end{proof}

\subsection{Simplicial sewing lemma}
In this subsection we state the promised simplicial sewing lemma -- Proposition~\ref{prop:sewing}. 
This tool will be used frequently in our analysis, but as mentioned earlier, the proof is deferred to Section~\ref{sec:sewing}. 

For a function $\Xi: \mathfrak{X}^k(\mathfrak{K})  \to \mathbb{R}$, a simplex $\sigma\in \mathfrak{X}^k(\mathfrak{K})$ and a subdivision $\mathcal{K} | \sigma$ we introduce the following quantity which measures to what degree $\Xi$ fails to be additive over $\mathcal{K}$:
\begin{equ}
\delta_{\mcK;\sigma} \Xi := \Xi(\sigma)-\sum_{\sigma' \in \mcK} \Xi(\sigma')\;.
\end{equ}
We now introduce a space of ``almost additive'' $\Xi$ which will be the domain of our sewing map.
\begin{definition}\label{def:germ}
Let $C_{2,k}^{{\eta,\gamma}}(\mathfrak{K})$ consist of all functions
$\Xi: \mathfrak{X}^k(\mathfrak{K})  \to \mathbb{R}$ satisfying  $\Xi(\sigma)=- \Xi(-\sigma)$ and
\[
\llbracket \Xi\rrbracket_{(\eta,\gamma);\mathfrak{K} }:= \llbracket \Xi\rrbracket_{\eta,\mathfrak{K}} +\llbracket \delta\Xi\rrbracket_{\gamma,\mathfrak{K}} <+ \infty\;,
\] 
%\ajay{Made $k$ explicit, that is wrote $C_{2,k}^{{\eta,\gamma}}(\mathfrak{K})$ in place of $C_{2,k}^{{\eta,\gamma}}(\mathfrak{K})$.}
where
\[
\llbracket  \Xi\rrbracket_{\eta,\mathfrak{K}}:= \sup_{\sigma \in \mathfrak{X}^k(\mathfrak{K})} \frac{|\Xi(\sigma)|}{\diam(\sigma)^{\eta} }\;,
\qquad
\llbracket  \delta\Xi\rrbracket_{\gamma,\mathfrak{K}}  := \sup_{\sigma \in \mathfrak{X}^k(\mathfrak{K})} \sup_{\mathcal{K}|\sigma} \frac{| \delta_{\mcK;\sigma} \Xi  |}{|\mathcal{K}| \diam(\sigma)^{\gamma}}\;,
\]
and the supremum over $\mathcal{K} | \sigma$ in the second expression runs over all subdivisions $\mathcal{K}$ of $\sigma$.
%Let $C^{\eta,\gamma}(\mathfrak{K})$ consist of all such functions
%$\Xi: \mathfrak{X}^k(\mathfrak{K})  \to \mathbb{R}$ satisfying  $\Xi(\sigma)=- \Xi(-\sigma)$ and
%\[
%\| \Xi\|_{\eta,\mathfrak{K}} +\| \delta\Xi\|_{\gamma,\mathfrak{K}} <+ \infty\;.
%\]
%%\red{Similarly, we define $C_{\diam}^{\eta,\gamma}(\mathfrak{K})$ to consist those functions }
\end{definition}

\begin{prop}\label{prop:sewing}
Let $0<\eta\leq k <\gamma$. 
There exists a unique linear map
$\mathcal{I}: C_{2,k}^{\eta, \gamma}(\mathfrak{K}) \to \Omega^k(\mathfrak{K}),$
satisfying
 \begin{equ}\label{eq:sewing_estimates}
 |\mathcal{I} \Xi (\sigma)- \Xi (\sigma)|\lesssim \llbracket \delta\Xi\rrbracket_{\gamma,\mathfrak{K}} \diam(\sigma)^{\gamma}, \qquad 
\big|\mathcal{I} \Xi(\sigma) \big| \lesssim \llbracket \Xi\rrbracket_{(\eta,\gamma); \mfK}  \diam(\sigma)^\eta \;,
\end{equ} uniformly in $\sigma \in \mathfrak{X}_{\le 1}^{k}(\mathfrak{K})$. 
Furthermore, for any regular sequence of subdivisions $(\mathcal{K}_n)_{n \in \mathbb{N}}$ of $\sigma\in \mathfrak{X}^{k}(\mathfrak{K})$, 
$$\mathcal{I}\Xi(\sigma)= \lim_{n\to \infty} \sum_{\sigma'\in \mathcal{K}_n} \Xi(\sigma')\ .$$
\end{prop}

\begin{remark}
Note that the condition $\Xi(\sigma)=- \Xi(-\sigma)$ in Definition~\ref{def:germ} could be replaced by $|\Xi(\sigma)+ \Xi(-\sigma)|=o(\diam(\sigma)^k)$.
If one did not impose any relation between $\Xi(\sigma)$ and $\Xi(-\sigma)$ an analogue of Proposition~\ref{prop:sewing} still holds, but $\mathcal{I}\Xi$ would only be additive but not an element of $\Omega^k$ in general.
\end{remark}
{
\begin{remark}\label{rem:disclaimer}
In this remark we compare the conditions of our sewing lemma (Proposition~\ref{prop:sewing}) to the conditions of \cite[Lem.~14]{Harang2018}, ignoring the distinction between simplices and hyperrectangles. 
The conditions therein are that for $\alpha\in (0,1]^k, \beta\in (1, \infty)^k$,
$$|\Xi(Q)|\lesssim h_1^{\alpha_1}\cdot...\cdot h_k^{\alpha_k}\ , \qquad |\Xi(Q)-\Xi(Q^{j}_1)-\Xi(Q^{(j)}_2)|\lesssim h_1^{\alpha_1}\cdot...\cdot h_j^{\beta_j}\cdot ... h_k^{\alpha_k}\ ,$$
uniformly over hyperrectangles $Q$ with dimensions $h_1,...,h_k$ (i.e. $Q=x + [0,h_1]\times...\times[0,h_k]$ for some $x\in \mathbb{R}^k$) and any decomposition 
$Q= Q^{(j)}_1\cup Q^{(j)}_2$ into hyperrectangles by cutting along a hyperplane parallel to $\{x_j=0\}$ for any $j\in \{1,...,k\}$.
Formally translating our bounds into hyperrectangle results in
$$ |\Xi(Q)|\lesssim (\max_i h_i)^\eta \ , \qquad |\Xi(Q)-\Xi(Q^{j}_1)-\Xi(Q^{(j)}_2)|\lesssim (\max_i h_i)^{\gamma}\ .$$
We see that the bound we impose is much weaker on cubes where $\min_i h_i \ll \max_i h_i$, resulting in the higher required exponent $\gamma>k$ compared to $\min \beta_j>1$. 
Conceptually, the conditions in \cite{Harang2018} allow for a component by component approach to sewing, while our conditions require a genuinely $k$-dimensional argument.
% explaining the $k$-dependent condition on $\gamma$.
 
Another difference with regards to \cite{Harang2018} is that therein different directions are allowed to scale differently. 
We believe our results could be formulated on $\mathbb{R}^d$ equipped with a non-trivial scaling $\fraks$ (cf.\ \cite{FS82, Hai14, MS23}),but we refrain from doing in order to give a clearer exposition.
\end{remark}
}

We state the following corollary of Proposition~\ref{prop:sewing}. 

\begin{corollary}\label{cor:norm_estimate_sewing}
In the setting of Proposition~\ref{prop:sewing}, if $ \alpha:=\eta-k+1\in (0,1]$, 
$$ \big|\mathcal{I} \Xi(\sigma) \big| \lesssim \llbracket \Xi\rrbracket_{(\eta,\gamma); \mfK}  \mass_{{\alpha}}^k (\sigma)\ .$$
\end{corollary}
\begin{proof}
This follows immediately from Proposition~\ref{prop:sewing} and Lemma~\ref{lem:equivalent norms}
\end{proof}
\begin{remark}\label{rem:non-atomic}
In \cite{AST24}, the authors write \say{It is worth emphasizing that this result might seem counter-intuitive...It is therefore not at all clear how one can expect to be nonatomic a limit of sums of germs which are essentially not so.} 
The proof of the corollary above clarifies this point. 

We also note that the threshold $\eta > k-1$ above  is expected  to be sharp,  cf.\ Proposition~\ref{prop:isomorphism}.
\end{remark}

\section{Distributional $k$-forms}\label{sec:distributional}
Given an element $A\in \Omega^k(\mathfrak{K})$, we set $\partial A\in \Omega^{k+1}(\mathfrak{K})$ to be defined by
$\big(\partial A\big)(\sigma):=  A(\partial \sigma)$ for all $\sigma \in \mathfrak{X}^{k+1}(\mathfrak{K})$, where the $\partial \sigma\in \mathcal{X}^{k}(\mathfrak{K})$ was defined in \eqref{eq:boundary op}. We now turn to the main definition of the article.
\begin{definition}\label{def:dist_form}
For $\alpha\in (0,1]$ and $\mathfrak{K} \subset \mathbb{R}^{d}$, we define the following norm on $\Omega^k(\mathfrak{K})$ 
$$\|A\|_{\alpha; \mathfrak{K}}
	=
	\sup_{\osimplex\in \mathfrak{X}_{\le 1}^k(\mathfrak{K})}
	\frac{|A(\osimplex)|}{\mass^k_\alpha(\osimplex)} \ \qquad
\text{ and }
\qquad \|A\|_{\infty; \mathfrak{K}} = \begin{cases} +\infty &\text{ if } A\neq 0, \\
0 &\text{ else.}
\end{cases}
$$
For $\alpha,\beta \in [0,1]\cup \{\infty\}$ set $\|A\|_{\alpha,\beta; \mathfrak{K}}:= \|A\|_{\alpha; \mathfrak{K}} + \|\partial A\|_{\beta; \mathfrak{K}}$ and 
%$\Omega^{k}_{(\alpha,\beta)}$
%
%,\beta\in (0,1] \cup \{\infty\}$ and $A \in \Omega^k$
\[ 
\Omega^{k}_{\alpha,\beta} (\mathfrak{K}):= \Big\{ A \in \Omega^k \ : \  \|A\|_{\alpha,\beta; \mathfrak{K}}< +\infty  \Big\}
%
%	\|A\|_{\alpha,\beta; \mathfrak{K}}
%	=
%	\sup_{\osimplex\in \mathfrak{X}_{\le 1}^k(\mathfrak{K})}
%	\frac{|A(\osimplex)|}{\mass_\alpha(\osimplex)} 
%	+
%	\sup_{\osimplex\in \mathfrak{X}_{\le 1}^{k+1}(\mathfrak{K})}
%	\frac{|A(\partial\osimplex)|}{\mass_\beta(\osimplex)}
\;.
\]
\end{definition}
\begin{remark}\label{rem:holder}
Note that $\Omega^{0}_{\alpha,\beta}=\Omega^{0}_{0,\beta} $ is canonically identified with $\beta$-H\"older functions, whenever $\beta \in (0,1]$. 
Also observe the inclusions $\Omega^k_{\alpha', \beta'} \subset \Omega^k_{\alpha, \beta} $ whenever $\alpha'\geq \alpha$, $\beta' \geq \beta$.

By construction the linear map 
\begin{equ}\label{eq:partial}
\partial: \Omega^{k}_{\alpha,\beta}\to \Omega^{k+1}_{\beta,\infty}\;.
\end{equ}
is bounded and 
the spaces $A\in\Omega^k_{\alpha,\infty}$ are a natural analogue of closed forms.
\end{remark}

When $\mathfrak{K} = \mathbb{R}^{d}$ we often suppress it from the notation, for instance writing $\|A\|_{\alpha}$ for $\|A\|_{\alpha, \mathbb{R}^{d}}$.
We also state most of our estimates below in terms of global norms $\| \cdot \|_{\alpha}$ to lighten notation, 
but it will be clear how they can be localized to compact $\mathfrak{K} \subset \mathbb{R}^{d}$.

\begin{remark}\label{rem:gauge}
When taking $k=1$ and $n=2$, the first term in the above definition corresponds to the growth norm $| \cdot |_{\gr\alpha}$ as in \cite[Def.~3.7]{CCHS22}, and the second term serves a similar purpose as $| \cdot |_{\tri\alpha}$ in \cite[Def.~3.10]{CCHS22}.
\end{remark}

\subsection{Estimates on chains and a generalised flat norm}\label{sec:estimates on chains}

In order to do analysis on $\mathcal{X}^k$, we introduce, for $X\in \mathcal{X}^k$, 
$$
| X |_{\alpha;k}:= \inf 
\Big\{  \sum_{j \in J} |c_j|  \cdot \text{\textbf{m}}_\alpha^k(\simplex_j): 
X= \sum_{j \in J} c_j \osimplex_j ,\ \osimplex_j \in \mathfrak{X}_{\le 1}^{k} \Big\}\;. 
$$
Clearly $| \cdot |_{\alpha;k}$ satisfies the triangle inequality. 

%\gray{\begin{remark}\ajay{Check this. Before only claimed $ | \sigma |_{\alpha} \le {\mathbf{m}}_\alpha^k(\sigma)$ for simplices - not needed yet but seems good to state if true}
%It is also straightforward to check that
%\begin{equ}\label{generalised mass}
%| X |_{\alpha;k} 
%=
%\inf 
%\Big\{  \sum_{j \in J} |c_j|  \cdot {\mass}_\alpha^k(\simplex_j): 
%X= \sum_{j \in J} c_j \osimplex_j ,\ \osimplex_j \in \mathfrak{X}_{\le 1}^{k},\; 
%\text{ with } \mathring{\simplex}_{j} \text{ disjoint } \Big\}\;. 
%\end{equ} \harprit{Remove or prove gray?}
%Therefore, for $\sigma \in \mathfrak{X}^k$ seen as an element of $\mathcal{X}^k$, Lemma~\ref{lem:(alpha,k) mass} implies that 
%$ | \sigma |_{\alpha} = {\mathbf{m}}_\alpha^k(\sigma)\;.$
%\end{remark}
%}
 
Using the above control, we now define a ``H\"{o}lder flat-norm'' on $\mathcal{X}^{k}(\mathfrak{K})$. 
\begin{definition}\label{def:flat_norm}
For $\alpha,\beta \in [0,1]\cup\{\infty\}$ define the $(\alpha,\beta)$-flat norm on $\mathcal{X}^{k}(\mathfrak{K})$ as
\begin{equ}\label{eq:flat_norm}
|X|_{(\alpha,\beta);\mathfrak{K}}
= 
\inf_{Z \in \mathcal{X}^{k+1}(\mathfrak{K})}
\Bigg\{
|X- \partial Z|_{\alpha;k} 
+ 
| Z |_{\beta;k+1}  
\Bigg\}\;.
\end{equ}
\end{definition}

\begin{remark}\label{rem:generalised flat norm}
The definition \eqref{eq:flat_norm} for $\alpha=\beta=1$ is  equivalent to the ``flat-norm'' introduced in \cite{Whi12}. For $\alpha=1$, $0<\beta< 1$ it is seen to be equivalent to the $(k-1+\beta)$-flat norm introduced in \cite{GaussGreen} using Lemma~\ref{lem:diam_mass characterisation}.
\end{remark} 

It is clear $|\cdot|_{(\alpha,\beta);\mathfrak{K}}$ satisfies the triangle inequality, that is for  $X,Y\in \mathcal{X}^{k}(\mathfrak{K})$
\begin{equ}\label{weak triangle inequality}
|X+Y|_{(\alpha,\beta);\mathfrak{K}}\leq |X|_{(\alpha,\beta);\mathfrak{K}} + |Y|_{(\alpha,\beta);\mathfrak{K}}\;.
\end{equ}

Similarly, for $Z \in  \mathcal{X}^{k+1}(\mathfrak{K})$, 
\begin{equ}\label{weak triangle2}
|X|_{(\alpha,\beta);\mathfrak{K}}\leq |X+\partial {Z}|_{(\alpha,\beta);\mathfrak{K}} + |Z|_{\beta, k+1}\;.
\end{equ}

\begin{lemma}\label{lem:pairing_bound}
For $A\in \Omega^k$ we have that 
$$\sup_{X\in \mathcal{X}^{k}(\mathfrak{K}) }\frac{|A(X)|}{|X|_{(\alpha, \beta);\mathfrak{K}}} \leq \|A\|_{(\alpha,\beta); \mathfrak{K}} \ . $$
\end{lemma}
\begin{proof}
For $X\in \mathcal{X}^k$ and any decomposition $X= \sum_i c_i \sigma_i$
$$|A(X)|\leq \sum_i |c_i| |A(\sigma_i)| \leq \|A\|_{\alpha} \sum_i |c_i| \mass_\alpha^k (\sigma_i) $$
and thus taking an infimum over all decompositions gives $|A(X)|\leq \|A\|_{\alpha} |X|_\alpha$. 
Similarly, for any $Z\in \mathcal{X}^{k+1}$ 
$$|A(\partial Z)| \lesssim \|\partial A\|_{\beta} | Z |_{\beta}\ .$$
Thus, 
$$|A(X)|\leq |A(X + \partial Z)| + |A(\partial Z)| \leq  \|A\|_{\alpha} |X+ \partial Z|_\alpha  + \|\partial A\|_{\beta} | Z |_{\beta} \leq \|A\|_{(\alpha,\beta)}  \left(  |X+ \partial Z|_\alpha  +| Z |_{\beta} \right)\ .  $$
Taking the infimum over $Z \in \mathcal{X}^{k+1}(\mathfrak{K})$ concludes the claim.
\end{proof}

%\begin{lemma}
%I expect 
%$$\sup_{X\in \mathcal{X}^{k}(\mathfrak{K}) }\frac{|A(X)|}{|X|_{(\alpha, \beta);\mathfrak{K}}} \sim \|A\|_{(\alpha,\beta)} \ .$$
%\end{lemma}
%\begin{proof}
%Needs solution to plateau problem for simplices
%\end{proof}

\begin{corollary}\label{lem:its a norm}
$|\cdot |_{(\alpha,\beta)}: \mathcal{X}^k \to \mathbb{R}$ is a norm.
\end{corollary}
\begin{proof}
Consider $0\neq X\in \mathcal{X}^k$. Without loss of generality we can write for some $n\in \mathbb{N}$ 
$$X= \sum_{i=0}^n c_i \sigma_i$$ where $\Vol^{k}(\sigma_i\cap \sigma_j)=0$ whenever $i\neq j$.
Note that smooth differential forms are contained in $\Omega^{k}_{1,1}\subset \Omega^{k}_{\alpha,\beta}$, thus one easily constructs $A\in  \Omega^{k}_{\alpha,\beta}$ such that $|A(X)|>0$. In view of Lemma~\ref{lem:pairing_bound}, this implies $|X|_{(\alpha,\beta)}>0.$
\end{proof}

\begin{definition}\label{def:B^k}
Let $\mathcal{B}^{k}_{\alpha,\beta}$ be the completion of the normed $\mathbb{Z}$-module $(\mathcal{X}^k, |\cdot|_{(\alpha,\beta)})$.
\end{definition}
\begin{corollary}\label{cor:pairing}
The pairing
$\mathcal{X}^k \times \Omega^k_{\alpha,\beta}\to \mathbb{R}$ continuously extends to a map $\mathcal{B}_{\alpha,\beta}^{k} \times \Omega^k_{\alpha,\beta}\to \mathbb{R}$.
\end{corollary}
\begin{proof}
Lemma~\eqref{lem:pairing_bound} shows that elements of $\Omega^k_{\alpha,\beta}$ are elements of the continuous dual of 
$(\mathcal{X}^k, |\cdot|_{(\alpha,\beta)})$.
\end{proof}

Definition~\ref{def:B^k} together with Corollary~\ref{cor:pairing} gives a 
geometric integration theory similar to \cite{Whi12,GaussGreen}, see also more generally \cite{har98, har_notes}. For $\alpha=1$, $\beta\in (0,1]$ it agrees with the one in \cite{GaussGreen}.

\begin{remark}\label{rem: Manifolds in B}
By covering a closed manifold by images of (say) squares and approximating as in \cite[Sec.~2.1, Ex.~4]{har98} one sees that closed, oriented, $C^{1,\eta}$-manifolds are canonically contained in $\mathcal{B}_{\alpha,\beta}^{k}$ as soon as $\eta>0$. 
Similarly to the proof of Lemma~\ref{lem:mass bound_1}, one can check that a compact oriented $C^{1,\eta}$-manifold with boundary is in $\mathcal{B}_{\alpha,\beta}^{k}$ as soon as $\alpha>\frac{1}{1+\eta}$, see also Theorem~\ref{thm:stokes} and Remark~\ref{rem:twoways}.  
\end{remark}
%\ajay{I think the fact that it is a norm essentially follows from (i) Lemma~\ref{lem:pairing_bound}, (ii) for $X \in \mathcal{X}^{k}$ with $X \not = 0$, there exists a a smooth $k$ form $f$ with $f(X) \not = 0$, and (iii) if $f$ is smooth, then is distributional form norms of $f$ are finite.}

%
%\gray{
%\subsection*{Outline}
%Maybe we should fix a map $S: \sigma \mapsto (\text{Partition of ($\sigma$)})$ which is regular, that is there exists $C>0$ such that 
%after applying $S$ $n$-times, the resulting subdivision $\mcK^{(n)}$ still satisfies $\frac{\diam (\sigma')^k}{\Vol^k(\sigma')}\leq C\frac{\diam (\sigma)^k}{\Vol^k(\sigma)} $ and proceed to prove everything "abstractly" using this map $S$?  See \cite{EG} for a very well behaved map $S$.
%
%\begin{remark}
%This allows to avoid working to with a specific subdivision, which quickly makes one enter combinatorial, rather than analytic considerations.
%Furthermore, for many natural subdivision, such as bysection at the maximal edge. it is simply not know if they are regular...
%\end{remark}
%^}

\subsection{Further $(\alpha,\beta)$-flat norm estimates}

We collect some Lemmata about the $(\alpha,\beta)$-flat norm which will be useful later on.
\begin{lemma}\label{lem:distance_simplices_only one varying vertex}
Let $v_0,...v_k\in \mathbb{R}^d$ and $v'_0\in \mathbb{R}^d$ all be contained in a ball of radius $r>0$. Then, 
$$
\Big|[v_0,v_1,...,v_k]-[v'_0,v_1,...,v_k] \Big|_{(\alpha,\beta) } \lesssim_k r^{k-1} d(v_0,v_0')^\alpha + r^{k}  d(v_0,v_0')^\beta \ .$$
\end{lemma}
\begin{proof}
Consider $Z= [v'_0,v_0,v_1,...,v_k]$ and write $X= [v_0,v_1,...,v_k]-[v'_0,v_1,...,v_k]$. 
Then 
\begin{align*}
|X|_{(\alpha,\beta) } &\leq |X- \partial  Z|_{\alpha;k} 
+ 
| Z |_{\beta;k+1}  \\
&\leq \sum_{j=1}^k \big| [v'_0,v_0,v_1,...,\hat{v}_j,...,v_k] \big|_{\alpha;k}  + | Z |_{\beta;k+1} \\
&\leq k   r^{k-1} d(v'_0,v_0)^\alpha +  r^{k}d(v'_0,v_0)^\beta  \ .
 \end{align*}

\end{proof}

\begin{lemma}\label{lem:distance between simplices}
Consider simplices $\sigma= [v_0,...,v_k]$, $\sigma'= [v'_0,...,v'_k]$ contained in a ball of radius $r>0$. Let 
$\mu=\max \big\{|v_i-v_i'| \ \in \mathbb{R} : i=0,...,k \big\}$.
Then, 
$$|\sigma-\sigma'|_{(\alpha,\beta)} \lesssim r^{k-1} \mu^\alpha   + r^{k}\mu^\beta . $$
In particular, there exists $Z \in \mathcal{X}^{k+1}$ such that, uniform in $r, \mu$, one has 
\[
\| \sigma-\sigma' - \partial Z \|_{\alpha;k} \lesssim r^{k-1} \mu^\alpha\ , \qquad
\| Z \|_{\beta;k+1} \lesssim r^{k} \mu^\beta \ .
\]
\end{lemma}
\begin{proof}
We define $\sigma_i= [v_0,...,v_{i-1},v'_i, ...,v'_k]$ and note that $\sigma=\sigma_{k+1}$ and $\sigma'=\sigma_{0}$.
The statement follows from the using \eqref{weak triangle inequality} to write 
$|\sigma-\sigma'|_{(\alpha,\beta) }\leq \sum_{i=0}^{k} |\sigma_i-\sigma_{i+1}|_{(\alpha,\beta)} $ and then applying Lemma~\ref{lem:distance_simplices_only one varying vertex}.

The desired $Z \in \mathcal{X}^{k+1}$ can be chosen to be $\sum_{i=0}^{k} Z_{i}$ where $Z_i$ is given as in the proof of Lemma~\ref{lem:distance_simplices_only one varying vertex} applied to $\sigma_{i}$ and $\sigma_{i+1}$.
\end{proof}

Fix a family of maps $(\pi_{n})_{n \in \mathbb{N}}$, with  $\pi_{n}: \mathbb{R}^{d}\to 2^{-n} \mathbb{Z}^{d} \subset \mathbb{R}^{d}$, such that we have
$$ d(x,\pi_{n}x) \lesssim 2^{-n}\ ,$$
uniformly in $n\in \mathbb{N}$ and $x\in \mathbb{R}^{d}$. 
We extend $\pi_n$ to a map on simplices by setting
$$\pi_{n} [v_0,v_1,...,v_k]= [\pi_{n}v_0,\pi_{n}v_1,...,\pi_{n}v_k] \;. $$ 
Note 
%that by linearity and since it behaves well with $\sim$, 
$\pi_{n}$ naturally extends to $\mathcal{X}^{k}$.
The statements of Lemma~\ref{lem:distance between simplices} specializes to the following corollary.

\begin{corollary}\label{cor:dyadic_diff_est}
For any  $\sigma \in \mathfrak{X}^{k}$ and $m,n \in \mathbb{N}$, 
$$|\pi_{m}\sigma-\pi_{n}\sigma|_{(\alpha,\beta)} \lesssim \diam(\sigma)^{k-1} 2^{- \alpha(m \wedge n)} + \diam(\sigma)^{k} 2^{- \beta(m \wedge n)}\ .$$
In particular, there exists $Z \in \mathcal{X}^{k+1}$ such that, uniform in $m,n$,
\[
| \pi_{m}\sigma-\pi_{n}\sigma - \partial Z |_{\alpha;k} \lesssim \diam(\sigma)^{k-1} 2^{- \alpha(m \wedge n)}\ , \qquad
| Z |_{\beta;k+1} \lesssim \diam(\sigma)^{k} 2^{- \beta(m \wedge n)} \ .
\]	
\end{corollary}

\section{Multiplication with regular functions}\label{sec:multiplication}
For a smooth function $f: \mathbb{R}^{d}\to \mathbb{R}$ and a smooth differential $k$-form $A$, one can easily see that the pointwise product $f\cdot A$ is  characterised by 
\begin{equ}\label{eq:def_product}
\int_{\sigma} f\cdot A = 
\lim_{n\to \infty} \sum_{\sigma'\in \mathcal{K}_n}  \Big( \frac{1}{k+1} \sum_{i=0}^k f(v^{\sigma'}_i)  \cdot \  \int_{\sigma'} A \Big) 
\end{equ}
where $(\mathcal{K}_n)_{n \in \mathbb{N}}$ is any regular sequence of subdivisions of $\sigma$ and $\sigma'=[v^{\sigma'}_0,...,v^{\sigma'}_k]$.
In order to define a product 
$$C^\beta \times \Omega^k_\alpha \to \Omega^{k}_\alpha, \qquad (f, A)\mapsto f\cdot A $$
we mimic the right-hand side of \eqref{eq:def_product}, but allow for slightly more general approximations.\footnote{
Note that $\Xi^{f\cdot A} [v_0,...,v_k]=\frac{1}{k+1} \sum_{i=0}^k f(v_i) A[v_0,...,v_k]$ is clearly a special case of \eqref{eq:prod_increment}.}
For this consider a family $\mu=\{\mu_{\sigma}\}_{\sigma\in \mathfrak{X}^k}$ of probability measures
 where each $\mu_{\sigma}$ is supported on $\pmb{\sigma}$ and independent of the orientation of $\sigma$, and set
\begin{equ}\label{eq:prod_increment}
\Xi^{f\cdot A}_\mu(\sigma):= \mu_\sigma(f)  \cdot A(\sigma) \;.
\end{equ}

\begin{theorem}\label{thm:multiplication}
Let $\alpha,\gamma \in (0,1]$ such that $\alpha+\gamma>1$ and $\mu=\{\mu_{\sigma}\}_{\sigma\in \mathfrak{X}^k}$ as above.  Then, $\Xi^{f\cdot A}_\mu\in C^{{k-1+\alpha},\gamma+ k-1 +\alpha}_{2,k}$ for every $f\in C^\gamma$, $A \in \Omega^{k}_{\alpha,\beta}$. 

It follows that, with the map $\mathcal{I}$ as in Proposition~\ref{prop:sewing}, 
$$f\cdot A:=\mathcal{I}\Xi^{f\cdot A}_\mu  \in \Omega^k_{\alpha, (\alpha+ \gamma-1) \wedge \beta}\;,$$
is in fact independent of the choice of $\mu$ and satisfies
\begin{equ}\label{eq:prod_bound}
\| f\cdot A\|_{\alpha} \lesssim \| f\|_{C^\gamma} \|A\|_{\alpha}, \qquad \| \partial (f\cdot A) \|_{(\alpha+ \gamma-1) \wedge \beta}\lesssim  \|f\|_{C^\gamma}\|A\|_{(\alpha,\beta)} \ .
\end{equ}
%In particular, $f\cdot A \in \Omega^k_{\alpha, (\alpha+ \gamma-1) \wedge \beta}\ .$
\end{theorem}

%\begin{remark}
%We can easily recover the Ito$=$Stratonivics theorem of \cite{AST24}. (TODO)
%It could also be interesting to check that their discrete approximation works in any dimension...
%\end{remark}
\begin{proof}
To verify that $\Xi^{f\cdot A}\in C^{\alpha,\gamma+ k-1+ \alpha}_{2,k}$, observe that
\[
 |\Xi^{f\cdot A}_\mu(\sigma)|\leq \|f\|_{L^\infty} |A(\sigma)| \leq \|f\|_{L^\infty} \|A\|_{\alpha} \mass_\alpha^k(\sigma)  \leq \|f\|_{L^\infty} \|A\|_{\alpha} \diam(\sigma)^{k-1+\alpha}\;,
 \]
  and 
\begin{align*}
|\delta_{\mcK;\sigma} \Xi^{f\cdot A}(\sigma)| &= \Big|\sum_{\sigma'\in \mathcal{K}} \left( \mu_{\sigma}(f) -\mu_{\sigma'}(f)   \right) A(\sigma') \Big|\\
& \leq |\mathcal{K}| \cdot \|f\|_{C^\gamma} \diam(\sigma)^\gamma \|A\|_{\alpha} \mass_\alpha^k(\sigma) \\
&\leq|\mathcal{K}| \cdot \|f\|_{C^\gamma}  \|A\|_{\alpha} \diam(\sigma)^{\gamma+ k-1+ \alpha}\ .
\end{align*}
The first estimate of \eqref{eq:prod_bound}  then follows from Corollary~\ref{cor:norm_estimate_sewing}. 

The fact that $f \cdot A$ does not depend on the particular choice of $\mu$ follows from the uniqueness in Proposition~\ref{prop:sewing} and the observation that, given a second family $\mu'=\{\mu'_\sigma\}_{\sigma\in \mathfrak{X}^k}$ of measures, 
$$|\Xi^{f\cdot A}_\mu(\sigma)- \Xi^{f\cdot A}_{\mu'}(\sigma)|\leq \|f\|_{C^\gamma} \diam(\sigma)^\gamma \|A\|_{\alpha} \mass_\alpha^k(\sigma)\;.$$ 

We now turn to the second inequality of \eqref{eq:prod_bound}. 
For a family of measures $\mu = \{ \mu_{w} \}_{w \in \mathfrak{X}^{k}}$ as before we define 
 $\Xi_{\mu}^{\partial (f\cdot A)}(\sigma) :=\sum_{F\in \Bd(\sigma)} \Xi_{\mu}^{f\cdot A}( F)$ for $\sigma \in \mathfrak{X}^{k+1}$. 

Given another family of measures on $k+1$ simplices
 $\bar{\mu}= \{\bar{\mu}_\sigma\}_{\sigma\in \mathfrak{X}^{k+1} }$, 
 we set, for $\sigma \in  \mathfrak{X}^{k+1}$. 
 $$\Xi^{ (\partial f)\wedge A}_{\mu, \bar{\mu}}(\sigma):= \sum_{F\in \Bd(\sigma)} \left(\mu_F(f)-\bar{\mu}_{\sigma}(f)\right) A(F)\ .$$
 Since
$\Xi^{\partial (f\cdot A)}_{\mu}= \Xi^{ (\partial f)\wedge A}_{\mu, \bar{\mu}} + \Xi^{f\cdot \partial A}_{\bar \mu}$, 
 \begin{align*}
|\Xi_\mu^{\partial (f\cdot A)}(\sigma)|&\leq \|f\|_{C^\gamma} \|A\|_{\alpha} \diam(\sigma)^{\gamma+k-1+ \alpha} + \|f\|_{L^\infty} \|\partial A\|_{\beta} \diam(\sigma)^{k+\beta} \\
&\leq \|f\|_{C^\gamma}\|A\|_{(\alpha,\beta)} \diam(\sigma)^{k + ( \gamma+ \alpha  -1 ) \wedge \beta}\;,\;
\end{align*} 
and thus
\begin{align*}
 &|\partial (f\cdot A) (\sigma )| \leq |\Xi_\mu^{\partial (f\cdot A)}(\sigma) - \partial (f \cdot A) (\sigma )| + |\Xi_\mu^{\partial (f\cdot A)}(\sigma)|\\
 &\leq 
(k+1) \llbracket \delta \Xi_{\mu}^{f\cdot A} \rrbracket_{\gamma+k-1+\alpha} \diam({\sigma})^{\gamma+k-1+\alpha}  + \|f\|_{C^\gamma}\|A\|_{(\alpha,\beta)} \diam(\sigma)^{k + ( \gamma+ \alpha  -1 ) \wedge \beta}
 \\
% & \leq 
%{k+1 \choose 2} \| \Xi^{f\cdot A}\|_{\gamma+k-1+\alpha} \diam({\sigma})^{\gamma+k-1+\alpha} 
%+  \|f\|_{C^\gamma} \|A\|_{\alpha} \diam^{\gamma+k-1+ \alpha} + |f|_{l^\infty} \|\partial A\|_{\beta} \diam^{k+\beta}\\
&\lesssim \|f\|_{C^\gamma}\|A\|_{(\alpha,\beta)} \diam(\sigma)^{k + ( \gamma+ \alpha  -1 ) \wedge \beta }   \ ,
\end{align*} 
which is the second inequality of \eqref{eq:prod_bound}.
In the first inequality above, we used that $|\Xi_\mu^{\partial (f\cdot A)}(\sigma) - \partial (f \cdot A) (\sigma)|
\le \sum_{F \in \Bd(\sigma)} |\Xi_\mu^{f\cdot A}(F) - (f\cdot A)(F)|$ and we have already estimates the terms on the right hand side. 
\end{proof}
The preceding theorem allows us to extend a wedge product on differential forms to our rough setting. 
\begin{corollary}\label{cor:wedge}
Let  $\alpha, \beta, \gamma \in (0,1]$ with $\tilde{\alpha}:=\alpha+\gamma-1>0$ and $\tilde{\beta}:=\beta+\gamma-1>0$.
Then  $df\wedge A:= d(f\cdot A) -f\cdot dA$ extends from smooth forms to a continuous bilinear map
\begin{equ}\label{eq:wedge}
 C^{\gamma}\times \Omega^{k}_{\alpha,\beta} \to \Omega^{k+1}_{\tilde{\alpha} \wedge \beta, \tilde{ \beta}}, \qquad (f, A)\mapsto   df\wedge A:= d(f\cdot A) 
-f\cdot dA \ 
\end{equ}
\end{corollary}
\begin{proof}
This follows from Theorem~\ref{thm:multiplication} which gives
$d(f\cdot A) \in \Omega^{k+1}_{(\alpha+\gamma -1)\wedge \beta, \infty}$ and $f\cdot dA \in  \Omega^{k+1}_{ \beta, \beta+\gamma-1}$. 
\end{proof}
We state another corollary, which in the view of Remark~\ref{rem:holder} and Corollary~\ref{cor:wedge}, reproduces \cite[Thm.~3.2]{Zus11} when specialized to $n=k$. 
\begin{corollary}\label{cor:zust}
Let $n\leq k$. 
Fix $\alpha,\gamma_0,...,\gamma_n\in (0,1]$ and $ \beta\in (0,1]\cup\{\infty\}$ such that 
\[
\alpha+\sum_{i=0}^{n} \gamma_i > n  \enskip \text{ and } \enskip \beta+ \sum_{i=0}^{n} \gamma_i >n-1\;.
\] 

Then, writing $\tilde\alpha=  \alpha+\sum_{i=1}^{n} \gamma_i -(n-1)$, $\tilde{\beta}=\big(\alpha + \sum_{i=0}^{n} \gamma_i-n \big)\wedge \big(\beta+\sum_{i=1}^{n} \gamma_i -(n-1) \big)$, one has that the mapping 
$$ C^{\gamma_0} \times \Big( \prod_{i=1}^{n} C^{\gamma_i} \Big) \times\Omega^{k-n}_{\alpha,\beta} \to \Omega^{k}_{\tilde \alpha,\tilde \beta} 
,\qquad (g_0,...,g_n, A) \mapsto g_0 \cdot dg_1 \wedge ...\wedge d g_n \wedge A $$
 is a bounded multi-linear map. 
\end{corollary}
\begin{proof}
We write $\alpha_{k-n} = \alpha, \beta_{k-n}= \beta$ and $A_{k-n}= A$. 
Then, for $l\in \{1,...,k\}$, inductively define $\alpha_{k-n+l} = (\alpha_{k-n+l-1} +\gamma_{k-l} -1)  \wedge  \beta_{k-n+l-1}$ and 
$ \beta_{k-n+l}= \beta_{k-n+l-1} +\gamma_{k-l} -1$. Note that $\alpha_{k-n+l}, \beta_{k-n+l}>0$ by assumption and thus
$A_{k-n+l}:= dg_n \wedge A_{k-n+l-1}\in \Omega^{k-n+l}_{\alpha_k-n+l  ,\beta_k-n+l} $ by \eqref{eq:wedge}.
For $l = n$, we have $A_{k}= dg_1 \wedge ...\wedge d g_n \wedge A$ and $\alpha_k=\tilde\alpha$ and  $\beta_k= \beta+\sum_{i=1}^{n} \gamma_i -(n-1)$.
Applying Theorem~\ref{thm:multiplication} once more to control $g_0 \cdot A_{k}$ completes the proof.
\end{proof}
{
\begin{remark}
Corollary~\ref{cor:zust} shows that our spaces $\Omega^k_{\alpha,\beta}$ contain the forms constructed by Z\"ust in \cite{Zus11}. 
\end{remark}
\begin{remark}
One can interpret Corollary~\ref{cor:zust} as giving a Young type wedge product between a sufficiently regular form $G=g_0 \cdot dg_1 \wedge ...\wedge d g_n$ and the rough form $A$. It would be interesting to find a sharp intrinsic characterisation of the required regularity on $G$, but this is beyond the scope of this work.
\end{remark}
}

%\begin{remark}\label{rem:relation to Züst}

% together 
%with the observation that $$d: C^{\gamma}\mapsto \Omega^{1}_{\gamma,\infty},\qquad g\mapsto \big([x,y]\mapsto g(y)-g(x) \big)$$ is bounded 

%it implies that for 
%$\gamma_0,...,\gamma_k\in (0,1]$ satisfying $\sum_{i=0}^{k} \gamma_i> k$ the map 
%$$ \prod_{i=0}^{k} C^{\gamma_i} \to \Omega^{k}_{\tilde \alpha,\tilde \beta} 
%,\qquad (g_0,...,g_k) \mapsto g_0 \cdot dg_1 \wedge ...\wedge d g_k \ $$ is bounded.
%\end{remark}

\section{Pullback and Stokes' theorem}\label{sec:pull-back}
%In this section we show that under appropriate assumptions elements of $\Omega^k $, can be pulled back. This provides intrinsic spaces $\Omega^k_{(\alpha,\beta)}(M)$. 

For $F : \mathbb{R}^{m}\to \mathbb{R}^{d}$ and a simplex $\sigma= [v_0,...,v_k] \in \mathfrak{X}^k(\mathbb{R}^m)$, write $F_*\sigma= [F(v_0),...,F(v_k)]$. 
For $A\in \Omega^k_{\alpha,\beta}$, we shall (under appropriate regularity assumptions) define
\begin{equ}\label{eq:def_pullback}
\big(F^{*}A\big)(\sigma):=\lim_{n\to \infty} \sum_{{\sigma'\in \mathcal{K}_n} } A(F_*\sigma')
\end{equ}
along any regular sequence of subdivisions $\mathcal{K}_n$ of $\sigma$. 
This is easily seen to agree with the classical definition whenever $F$ and $A$ are smooth.

\begin{theorem}\label{thm:pullback}
Let $\alpha\in ( \frac{1}{1+\eta} ,1]$, $\beta>0$. 
Then, for $A \in \Omega_{\alpha,\beta}$ and $F\in C^{1,\eta}$, 
the limit in \eqref{eq:def_pullback} is well defined and for $\tilde{\beta}=  \big(\alpha(1+\eta)-1\big)\wedge \beta$ the linear map
$$ F^{*}: \Omega^k_{\alpha,\beta} \to \Omega^k_{\alpha,\tilde \beta }, \qquad A\mapsto   F^*A \ $$ is bounded.

Furthermore, for $ \bar F \in C^{1,\eta}$, $F^*(\bar{F}^*A)=( \bar{F}\circ F)^* A$,  and, for $\phi \in C^{\gamma}$ with $\gamma\in (1-\alpha,1]$, $F^{*}(\phi \cdot A)= (F^*\phi)\cdot F^{*} A$. 

Finally, for $\tilde{\alpha}< \alpha$, $\tilde{\beta}<  \big(\alpha(1+\eta)-1\big)\wedge \beta$ the following map is continuous
$$  C^{1,\eta}\times \Omega^k_{\alpha,\beta}\to \Omega^k_{\tilde{\alpha},\tilde \beta}, \qquad (F,A)\mapsto   F^*A\ .$$  
\end{theorem}
\begin{proof}
Combine Lemma~\ref{lem:mass bound_1}, Lemma~\ref{lem:mass bound_2} and Proposition~\ref{prop:pullback2}.
\end{proof}
Note that this implies that for $\eta >0$, $\alpha>\frac{1}{1+\eta} , \beta <  \alpha(1+\eta)-1$, and an oriented $C^{1,\eta}$-manifold $M$  (possibly with boundary) we can define a space 
$\Omega_{\alpha,\beta}^k (M)$ in the usual way.
 In particular, this provides a notion of integration by pull-back and mimicking standard proofs of Stokes' theorem, cf.\ \cite{Lee10}, one obtains the following theorem. 
\begin{theorem}\label{thm:stokes}
Let $M$ be a
compact $(k+1)$--dimensional, oriented, $C^{1,\eta}$--manifold with boundary. 
For any $\alpha> \frac{1}{1+\eta}$, $\beta>0$, and $A\in \Omega^{k}_{\alpha,\beta}(M)$, it holds that 
$$ \int_M dA = \int_{\partial M} A \;. $$

%
%$\Omega^{k}_{(\alpha,\beta)}(N)$ forms can be integrated over compact $k$-dimensional oriented compact $C^{1,\eta}$-manifolds whenever $\alpha> \frac{1}{1+\eta}$ and $\beta>0$. Furthermore it holds that 
%$$ \int_N dA = \int_{\partial N} A $$
%for any compact 
%compact $k+1$-dimensional oriented $C^{1,\eta}$-manifold with boundary $N$.
\end{theorem}

%\begin{lemma}
%Let $\alpha,\beta, \gamma$, ... if $\|F\|_{...}<\infty$, $A\in \Omega_{\alpha,\beta}$.
%$\partial (F^* A)= F^*( \partial A)$
%\end{lemma}
%\begin{proof}
%For $\sigma\in \mathfrak{X}^{k+1}$
%$$ \partial F^{*}A(\sigma) = F^{*}A( \partial \sigma)   = A(F_* \partial \sigma) + \red{\mathcal{o}(\diam(\sigma)^k) }
%= A(\partial F_*  \sigma) + \mathcal{o}(\diam (\sigma)^k) = \partial A( F_*  \sigma) + \mathcal{o}(\diam (\sigma)^k) 
%= \partial A( F_*  \sigma) + \mathcal{o}(\diam (\sigma)^k) 
%$$
%and 
%$$  \partial A( F_*  \sigma) = F^* \partial A(   \sigma) + \mathcal{o}(\diam(\sigma)^{k+1})$$
%\end{proof}

\begin{remark}\label{rem:twoways}
Contrast the stronger regularity assumptions required to build integration by pull-back, Theorem~\ref{thm:pullback},  compared to Remark~\ref{rem: Manifolds in B}.
It is not hard to see that when both notions apply, the resulting integrals are equal. 
\end{remark}

\begin{remark}\label{rem:control}
In contrast to \cite[Theorem 3.18]{CCHS22} we do not work with a control, since it is unclear how to extend Definition~3.16 therein to $k>1$.
For $k=1$ (and $d\geq 2$) one can still define a control by replacing $|P;\bar{P}|^{\alpha/2}$ therein by the $(\alpha', \beta)$-flat norm for some $\alpha', \beta\in (0,1]$, but we shall not pursue this since our focus is on generic $k\geq 1$.
\end{remark}

\subsection{Simplicial approximation}

For $F: \mathbb{R}^m\to \mathbb{R}^{d}$ and  $\osimplex=[v_0,...,v_k]\in \mathfrak{X}^k(\mathbb{R}^m)$, let $F^{\sigma} : \mathbb{R}^m \to \mathbb{R}^d$
to be the unique map supported on $\pmb{\sigma}$ such that
 $F^{\sigma}(\sum_{i} t_i v_i)= \sum_{i} t_i F(v_i)$ whenever $t_i\in [0,1]$ and $\sum_{i} t_i=1$.
%For a set of simplices $\mathcal{K}$ set $F^{\mathcal{K}}= \sum_{\sigma\in\mathcal{K}} F^{\sigma} \ .$
%\gray{Similarly, we associate to $\mathcal{K}$ the set of simplices 
%$$F_*\mathcal{K}:=\left\{ F_*\sigma : \sigma\in \mathcal{K} \right\} \ . $$
%}
For a simplex $\osimplex\in \mathfrak{X}^k(\mathbb{R}^m)$ let
%Let $M= M(k)$ be large enough. \ajay{Note: Move largeness condition on $M$ to where it is used}
\begin{equs}
|F|_{(\alpha,\beta);\osimplex} &:=\sup_{\mathcal{K} | \sigma} \max_{\sigma'\in \mathcal{K}} |F_*\sigma'- F^{\sigma}_*\sigma'|_{(\alpha,\beta)} \ , \\
|F; \bar F|_{(\alpha,\beta);\osimplex} &:= \sup_{\mathcal{K} | \sigma} \max_{\sigma'\in \mathcal{K}} |F_*\sigma' -\bar F_* \sigma' - F^{\sigma}_*\sigma'+ \bar{F}^{\sigma}_*\sigma'|_{(\alpha,\beta)}\;.
\end{equs}
where the supremum runs over all subdivisions $\mathcal{K}$ of ${\osimplex}$. Let 
%\begin{equ}
%\|F\|_{(\alpha,\beta),\gamma; k ; \mfK}=  \sup_{\sigma' \in \mathfrak{X}^k_{\leq 1}( \mfK)} \left( \frac{ |F|_{(\alpha,\beta);\osimplex'}}{\diam(\sigma')^{\gamma}}   \right)\;, \qquad 
%\vertiii{F}_{\alpha; k; \mfK} := \sup_{\sigma' \in  \mathfrak{X}^k_{\leq 1}( \mfK)} \left( \frac{ \diam (F_*\sigma')}{\diam(\sigma')^{k-1+\alpha}}   \right)\;.
%%\red{ |\gamma|_{\bar \alpha, \alpha; \sigma}=  \sup_{\sigma' \subset \sigma} \left( \frac{ |\gamma|^{\bar \alpha}_{\osimplex'}}{\diam(\sigma')^{(1-\alpha/\bar{\alpha} )k}  \Vol^k(\sigma')^{\alpha}}   \right) } 
%\end{equ}
%and 
\begin{equs}
\|F;\bar F \|_{(\alpha,\beta);\eta; k; \mfK} :&= \sup_{\sigma \in   \mathfrak{X}^k_{\leq 1}( \mfK)} \left( \frac{|F_*\sigma- \bar{F}_*\sigma  |_{(\alpha, \beta)}}{\diam(\sigma)^{\eta}}   \right)\ ,  \\
\vertiii{F; \bar F }_{(\alpha,\beta);\gamma; k; \mfK}:&=  \sup_{\sigma\in   \mathfrak{X}^k_{\leq 1}( \mfK)} \left( \frac{ |F; \bar F |_{(\alpha,\beta);\osimplex}}{\diam(\sigma)^{\gamma}}   \right)\;, 
%\red{ |\gamma|_{\bar \alpha, \alpha; \sigma}=  \sup_{\sigma' \subset \sigma} \left( \frac{ |\gamma|^{\bar \alpha}_{\osimplex'}}{\diam(\sigma')^{(1-\alpha/\bar{\alpha} )k}  \Vol^k(\sigma')^{\alpha}}   \right) } 
\end{equs}
as well as $\|F\|_{(\alpha,\beta),\gamma; k ; \mfK}:= \|F;0 \|_{(\alpha,\beta),\gamma; k; \mfK}$ and $  \vertiii{F}_{(\alpha,\beta);k; \mfK} :=\vertiii{F; 0 }_{(\alpha,\beta);k; \mfK} $.
Note that it follows directly from 
\eqref{weak triangle inequality}
that $\|\ \cdot\ ;\ \cdot\ \|_{(\alpha,\beta);\eta; k; \mfK}$ and $\vertiii{\ \cdot\ ; \ \cdot\ }_{(\alpha,\beta);\gamma; k; \mfK}$ satisfy the triangle inequality. 
\begin{lemma}\label{lem:mass bound_1}
Whenever $F$ is Lipschitz continuous, for every $R>0$ it holds that
\begin{equ}\label{eq:former_estimate_revision}
\|F\|_{(\alpha,\beta); k-1+\alpha; k; \mfK} \lesssim_{R} \| DF\|_{L^{\infty}}\ ,
\end{equ}
uniformly over $F$ satisfying $\| DF\|_{L^{\infty}}\leq R$.
Furthermore, for $\eta<k-1+\alpha$ there exists $\kappa= \kappa(\alpha,\eta)>0$ such that for every $R>0$
\begin{equ}\label{eq_local}
\|F; \bar F\|_{(\alpha,\beta); \eta; k; \mfK} \lesssim_R \|F-\bar{F}\|^{\kappa}_{L^\infty} 
\end{equ}
uniformly over Lipschitz continuous $F,\bar F$ such that $\| F\|_{L^{\infty}},\ \| \bar F\|_{L^{\infty}}, \ \| D F\|_{L^{\infty}} ,\  \| D\bar F\|_{L^{\infty}}\leq R$.
 \end{lemma}
\begin{proof}
The former estimate \eqref{eq:former_estimate_revision} is seen directly. Next, note that on the one hand 
$$|F_*\sigma- \bar{F}_*\sigma  |_{(\alpha, \beta)} \leq |F_*\sigma |_{(\alpha, \beta)}+| \bar{F}_*\sigma |_{(\alpha, \beta)}  \lesssim  {\diam(\sigma)^{k-1+\alpha}} $$
and on the other hand $|F_*\sigma- \bar{F}_*\sigma  |_{(\alpha, \beta)}\lesssim \|F-\bar{F}\|_{L^\infty}^{\alpha}$ by Lemma~\ref{lem:distance between simplices}. 
Therefore, for any $\lambda\in (0,1)$
$$|F_*\sigma- \bar{F}_*\sigma  |_{(\alpha, \beta)} \lesssim {\diam(\sigma)^{(k-1+\alpha)\lambda}} \|F-\bar{F}\|_{L^\infty}^{(1-\lambda)\alpha}\ ,$$
which implies \eqref{eq_local} by choosing $\lambda<1$ sufficiently large.
\end{proof}

\begin{lemma}\label{lem:mass bound_2}
Let $\alpha\in ( \frac{1}{1+\eta}, 1]$, $\beta\in( 0,1]$, then $ \bar{\gamma}:= \left(k-1+ \alpha(1+\eta) \right) \wedge\left( k+ \beta(1+\eta)\right)>k$.
For every $R>0$ and $\gamma\leq \bar{\gamma}$ the bound
$\vertiii{F }_{(\alpha,\beta),\gamma;k }\lesssim_{R} \|F\|_{C^{1,\eta}}$ holds
uniformly over $F$ satisfying $\|F\|_{C^{1,\eta}}<R$. 

Furthermore, if $\gamma< \bar{\gamma}$, 
there exists $\kappa=\kappa(\alpha,\beta,\gamma,\eta)>0$ such that
$\vertiii{F; \bar F }_{(\alpha,\beta),\gamma}\lesssim_{R} \|F-\bar F\|^{\kappa}_{L^\infty}$ uniformly over $F,\bar F$ satisfying $\|F\|_{C^{1,\eta}}, \|\bar F\|_{C^{1,\eta}}<R$.
 \end{lemma}
%
%\begin{remark}
%\gray{
%Note that we will only be interested in $\|\cdot; \cdot\|_{(\alpha,\beta),\gamma ;k;\mcK}$ for $\gamma>k$, which in the setting of Lemma~\ref{lem:mass bound_1} translates into the conditions $\alpha>\frac{1}{1+\eta}$, $\beta>0$.}
%\end{remark} 
% 
% \red{
% \begin{lemma}\label{lem:mass bound_1}
%For every $R>0$ and $\bar \alpha\in [ \frac{k}{k+\eta}, 1]$ and $\alpha \in [0, \bar{\alpha}]$ the bound
%$|\gamma |_{\bar \alpha, \alpha }\lesssim_{R} \|\gamma\|_{C^{1,\eta}}$ holds
%uniformly over $\gamma$ satisfying $\|\gamma\|_{C^{1,\eta}}<R$. 
%Furthermore,  $|\gamma|_{\alpha}<\infty$ for every smooth $\gamma$ if and only if $\alpha \in [ \frac{k}{k+1}, 1]$. 
% \end{lemma}}
\begin{proof}
%We first prove that $\alpha \geq  \frac{k}{k+\eta}$ implies that, for every $\gamma\in C^{1,\eta}$, one has $| \gamma |_{\alpha} < \infty$. 
%Let $\sigma' \subset \sigma$, then recall that 
%$$ |\gamma^{\sigma}(\sigma')- \gamma(\sigma')|_{\alpha,\beta;\sigma} = \inf_{Z\in \mathcal{X}^k} \left(\left|\gamma^{\sigma}(\sigma')- \gamma(\sigma')-Z\right|_{\alpha;k} + \left|Z\right|_{\beta}\right)\ .$$
Let $\sigma' \subset \sigma$ and observe that each vertex $v'_i\in \VV(\sigma')$ can uniquely be written as a convex combination
$ v'_i= \sum_{j=0}^k \lambda^i_j v_j$ where $v_j\in\VV(\sigma)$. 

%Thus
%$$
%F (\sigma') 
%%= \big \{ \sum_{i=0}^k \lambda_i F (v'_i) \ :\lambda \in \triangle_k \big\}
%=[F (v'_0), F (v'_1), ...,F (v'_k)]\;,
% \qquad F^\sigma (\sigma')
%= 
%%\big\{ \sum_{i=0}^k \lambda_i F^{\sigma} (v'_i) \ :\lambda \in \triangle_k \big\}\\
%%=
%%\big\{ \sum_{i,j=0}^k \lambda_i  \lambda^i_j F(v_j)  \ : \lambda \in \triangle_k  \big\}
%%=
%[\sum_{j}\lambda^0_j F(v_j), ...,\sum_{j}\lambda^k_j(v_j)]
%\ .$$
Thus we observe that for $r= \diam (\sigma)$, 
\begin{equs}
\sum_j \lambda^i_j F(v_j) =& \sum_j \lambda^i_j \left( F (v'_i) + \nabla  F (v'_i) (v_j -v'_i)  +\mathcal{O}(\|D^1F\|_{C^\eta}(v_j -v'_i)^{1+\eta}) \right) \\
&= F (v'_i) + \nabla  F (v'_i)   \sum_j \lambda^i_j (v_j -v'_i)  +\mathcal{O}(\|D^1F\|_{C^\eta}r^{1+\eta})\\
&= F (v'_i)+ \mathcal{O}(\|D^1F\|_{C^\eta}r^{1+\eta}) \ . \label{eq:bigO}
\end{equs}
Since both $F^\sigma (\sigma')$ and  $F (\sigma')$ are contained in a ball with radius $\lesssim \|DF\|_{L^\infty}r$,
%Thus, using \eqref{weak triangle inequality}
%$$|\gamma^{\sigma}(\sigma')- \gamma(\sigma')|_{(k+1)M} \lesssim \sum_{l=0}^k \left| [\gamma (v_0),..., v_{l-1},\sum_{j}\lambda^l_j \gamma(w_j), ...,\sum_{j}\lambda^k_j(w_j)]  -  [\gamma (v_0),..., v_{l},\sum_{j}\lambda^{l+1}_j \gamma(w_j), ...,\sum_{j}\lambda^k_j(w_j)] \right|_M \ . $$
we conclude  by Lemma~\ref{lem:distance between simplices} that
$$|F^{\sigma}(\sigma')- F(\sigma')|_{(\alpha,\beta)}  
\lesssim (\|DF\|_{L^\infty}r )^{k-1} (\|DF\|_{C^\eta}r)^{\alpha (1+\eta)}+(\|DF\|_{L^\infty}r)^{k}  (\|DF\|_{C^\eta}r)^{\beta(1+\eta)}\ , $$
which remains bounded by $r^{\gamma}$ 
whenever $\gamma\leq\bar{\gamma}$.
%$ \alpha\geq \frac{\gamma-k+1}{1+\eta} $ and $ \beta \geq \frac{\gamma-k}{1+\eta} \ .$

To see the second claim of the lemma, note that as above
\begin{equs}
|F^{\sigma}(\sigma')- F(\sigma')-\bar{F}^{\sigma}(\sigma')+ \bar{F}(\sigma')|_{(\alpha,\beta)}
& \leq  |F^{\sigma}(\sigma')-\bar{F}^{\sigma}(\sigma')  |_{(\alpha,\beta)} + | F(\sigma') -\bar{F}(\sigma')|_{(\alpha,\beta)} \\
&\lesssim_R \|F-\bar{F}\|_{L^\infty}\ . \label{eq:revision1}
\end{equs} 
%since  $F^{\sigma}(\sigma'), \bar{F}^{\sigma}(\sigma') $ are contained in a ball of radius $r\lesssim (|F|_{C^1} \vee  |\bar F|_{C^1}) $.
%each term is bounded 
%$$|F^{\sigma}(\sigma')- F(\sigma')-\bar{F}^{\sigma}(\sigma')+ \bar{F}(\sigma')|_{(\alpha,\beta)}\lesssim_R \|F-\bar{F}\|_{L^\infty}\ $$ 
%and thus $|F, \bar F|_{\osimplex}\lesssim_R \|F-\bar{F}\|_{L^\infty}$. 
On the other hand by the first part of this proof 
\begin{equ}\label{eq:revision2}
\vertiii{F; \bar F }_{(\alpha,\beta); \bar \gamma;k}\leq \vertiii{F }_{(\alpha,\beta);\bar \gamma;k} + \vertiii{ \bar F }_{(\alpha,\beta);\bar \gamma;k}< \infty  \ .
\end{equ} 
We
%
%$$\vertiii{F, \bar F}_{(\alpha,\beta)}\lesssim |F|_{(\alpha,\beta)} + |\bar F|_{(\alpha,\beta)} \lesssim_R (r^{k-1+\alpha (1+\eta)} + r^{k+ \beta(1+\eta)})\ , $$
 conclude by interpolation between \eqref{eq:revision1} and \eqref{eq:revision2} as at the end of the proof of Lemma~\ref{lem:mass bound_1}.
\end{proof} 

%\begin{itemize}
%\item\label{lem:easy_gamma_item2} $\|F;\bar F\|_{(\bar \alpha, \bar{\beta});\sigma } \leq \diam(\sigma)^{(\bar \alpha- \alpha) \wedge (\bar \beta- \beta)}  \|F;\bar F\|_{(\alpha, {\beta});\sigma }  $ for $\alpha\leq \bar{\alpha}$, $\beta\leq \bar{\beta}$.  \harprit{double check!}
%\end{itemize}

\begin{remark}
Note that the threshold $\alpha>1/2$ in Theorem~\ref{thm:pullback} is explained by the observation that even if $F$ is smooth, generically
$\vertiii{F}_{(\alpha,\beta); k;k}<\infty$ only if $\alpha \in [ \frac{1}{2}, 1]$. This can be seen by for example choosing
$F$ to be a smooth parametrisation of the closure of an open subset of an embedded $k$-sphere which makes  \eqref{eq:bigO} sharp. 
\end{remark}

In order to make sense of the limit in \eqref{eq:def_pullback}, we define the increments
$$\Xi^{F^*A}(\sigma):= A(F_*\sigma), \qquad \Xi^{F^*A ; \bar F^*A }(\sigma):= A(F_*\sigma)- A(\bar F_* \sigma)\;.$$

\begin{prop}\label{prop:pullback2}
Let $\eta>k-1$, $\gamma>k$. For $A \in \Omega^k_{\alpha,\beta}$ and $F,\bar{F}$ continuous such that 
$\| F\|_{( \alpha, \beta);\eta;k}, \| \bar F\|_{( \alpha, \beta);\eta;k}, \vertiii{F}_{(\alpha,\beta); \gamma;k} , \vertiii{\bar F}_{(\alpha,\beta); \gamma;k}  <\infty$, one finds that
$\Xi^{F^*A}, \Xi^{F^*A ; \bar F^*A }\in C_{2,k}^{\eta, \gamma}$
and that $F^*A:= \mathcal{I}\Xi^{{F^*A}} \in \Omega^k$ satisfies
\begin{equs}
\|F^*A\|_{\eta-(k-1) }&\lesssim \|A\|_{(\alpha,\beta)} \left( \| F\|_{( \alpha, \beta);\eta;k}  \vee \vertiii{F}_{(\alpha,\beta);\gamma;k} \right)\  ,\\
\|F^*A- \bar{F}^*A \|_{\eta-(k-1) } &\lesssim \|A\|_{(\alpha,\beta)}  \left( \| F;\bar F\|_{( \alpha, \beta);\eta;k}  \vee \vertiii{F;\bar F}_{(\alpha,\beta);\gamma;k} \right)\ .
\end{equs}
If furthermore $\vertiii{F}_{(\beta,\infty);\tilde{\gamma}; k+1}, \vertiii{F}_{(\beta,\infty);\tilde{\gamma}; k+1}<\infty$ for $\tilde{\gamma}>k$, then 
$F^*A\in \Omega_{\alpha, (\gamma\wedge \tilde \gamma)-k}\ $
and 
\begin{equ}\label{eq:derivative_pullback}
\|\partial (F^*A  -\bar F^*A)\|_{(\gamma \wedge \tilde{\gamma})- k }\lesssim  \vertiii{ F;\bar F}_{( \alpha, \beta);\gamma;k}  \vee\|F;\bar{F}\|_{(\beta,\infty);\tilde\gamma; k+1}\ .
\end{equ}
%If furthermore $\beta >\frac{1}{1+\eta}$, then $F^*A\in \Omega^k_{\alpha,\beta}$ and 
%it holds that $\partial F^*A= F^* \partial A \ .$
Finally, $F^*(\bar{F}^*A)=( \bar{F}\circ F)^* A$ and $F^{*}(\phi \cdot A)= F^*\phi\cdot F^{*} A$ for $\phi \in C^\rho$ whenever $\rho \in (1-\alpha,1]$.
\end{prop}

\begin{proof}
First note that by Lemma~\ref{lem:pairing_bound}
\begin{equ}\label{eq:local increment}
\llbracket\Xi^{F^*A }\rrbracket_\eta \leq    \|A\|_{ \alpha} \| F\|_{( \alpha, \beta);\eta;k} \ , \qquad \llbracket\Xi^{F^*A ; \bar F^*A}\rrbracket_\eta \leq \|A\|_{ (\alpha, \beta)}\| F;\bar F\|_{( \alpha, \beta);\eta;k}   \;.
\end{equ} 
To see
\begin{equ}\label{eq:local increment2}
\llbracket \delta \Xi^{F^*A}\rrbracket_{\gamma} \leq \|A\|_{(\alpha,\beta)} \vertiii{ F}_{( \alpha, \beta);\gamma;k}   \;,
 \qquad 
  \llbracket \delta_{\sigma, \mathcal{K}} \Xi^{F^*A ; \bar F^*A} \rrbracket_\gamma \leq \|A\|_{(\alpha, \beta)} \vertiii{ F; \bar{ F}}_{(\alpha,\beta);\gamma;k}  \end{equ}
note that for $\mathcal{K}|\sigma$
\begin{align*}
|\delta_{\sigma, \mathcal{K}} \Xi^F|&= \left| 
A(F_* (\sigma)) - \sum_{\sigma'\in \mathcal{K} } A(F_* (\sigma')) \right| \leq \sum_{\sigma'\in \mathcal{K} } |A (F^{\sigma}_*\sigma'- F_*\sigma')| \\
&\leq  \|A\|_{(\alpha, {\beta}) }  \sum_{\sigma'\in \mathcal{K} }  |F^{\sigma}_*\sigma'- F_*\sigma'|_{( \alpha, \beta)} \\
&\leq  \|A\|_{(\bar{\alpha}, \bar{\beta})}   |\mathcal{K}| \vertiii{F}_{( \alpha, \beta) ;\gamma; k} \diam(\sigma)^\gamma \ .
\end{align*}
 The second inequality of \eqref{eq:local increment2} follows similarly. Thus, we conclude the first part of the proposition by Corollary~\ref{cor:norm_estimate_sewing}.
 
To see the second part, note that 
$$ |A ( \partial F_* \sigma) |  \leq \| \partial A \|_{\beta} \|F\|_{(\beta,\infty); \tilde{\gamma}; k+1 }\diam(\sigma)^{\tilde{\gamma}}$$
and 
$$|F^*A( \partial \sigma) - A ( \partial F_* \sigma) |\leq  \sum_{B\in \Bd(\sigma)}  | \Xi^{F^*A} (B) - F^* A(B) |   \lesssim \llbracket \delta \Xi^{F^*A}\rrbracket_{\gamma} \diam(\sigma)^{\tilde{\gamma}} \ .  $$
Therefore by \eqref{eq:local increment2}
\begin{align*}
 |F^*A( \partial \sigma) |&\leq  |A ( \partial F_* \sigma) | + |F^*A( \partial \sigma) - A ( \partial F_* \sigma) | \\
 &\lesssim  \|A\|_{(\alpha,\beta)} (\vertiii{ F}_{( \alpha, \beta);\gamma;k} \vee \|F\|_{(\beta,\infty); \tilde{\gamma}; k+1 })\diam(\sigma)^{\tilde{\gamma}\wedge\gamma} \ . 
\end{align*}
Thus, by Lemma~\ref{lem:equivalent norms}
$$\|\partial F^*A \|_{(\gamma \wedge \tilde{\gamma})- k }\lesssim  \|A\|_{(\alpha,\beta)} (\| F\|_{( \alpha, \beta);\gamma} \vee \|F\|_{(\beta,\infty); \tilde{\gamma}; k+1; }) \ $$
 and \eqref{eq:derivative_pullback} follows similarly.

%Since
%
%and similarly 
%$$ \Big| \sum_{F\in \Bd(\sigma)}\Xi^{F^*A ; \bar F^*A}(F) \Big| \leq \| A \|_{\beta} \vertiii{F;\bar{F}}_{(\beta,\infty), k+1}\diam(\sigma)^{k+\beta}\ .$$
 
% \gray{
%  Thus we conclude $F^*A:=\mathcal{I}\Xi^{F^*A}  \in \Omega^k$ extists and satisfies 
%$\|F^*A- \bar{F}^*A \|_\alpha\leq\| A\|_{\alpha,\beta} \|F; \bar F \|_{(\alpha,\beta); k}$, and thus \eqref{eq:pull_back quantitative} follows be Proposition~\ref{prop:sewing}.
% 
% To turn to the second part of the claim, let $\Xi^{\partial (F^{*}A)}(\sigma)= A(F_* \partial \sigma)$ for $\sigma\in \mathfrak{X}^{k+1}$.
% Then,  $$\delta_{\sigma, \mcK}\Xi^{\partial F^{*}A} = A(F_* \partial \sigma) - \sum_{\sigma'} A(F_* \partial \sigma')=  \sum_{\sigma'} A \left(F^{\sigma}_* \partial \sigma -  F_* \partial \sigma' \right)
%= \sum_{\sigma'} \partial A \left(F^{\sigma}_*\sigma -  F_* \sigma' \right)
%$$
%
%And thus $\|\delta\Xi^{\partial F^{*}A}\|_\gamma \lesssim \|\partial A \|_\beta \|F\|_{(\beta,\infty),\gamma;k+1}$, which is finite for $\gamma\leq k+\beta(1+\eta)$.
%We conclude the proof by noting that $\Xi^{\partial F^*A}(\sigma)= A(F_* \partial \sigma)= \Xi^{ F^*A}  (\partial \sigma)$
%and 
%$$ 
%|A(F_* \partial \sigma)|= | \partial A (F_* \sigma)|\leq |\partial A|_{\beta} |\mass^k_\beta (F_* \sigma) \lesssim  |\partial A|_{\beta} \| DF\|_{L^\infty} |\mass^k_\beta (\sigma) \ .
%$$}
The final identities follow from the uniqueness part of the sewing lemma.
The first of these follows from the fact that $G_* (F_*(\sigma))= (G\circ F)_* \sigma$. 
For the second identity, note that $\Xi_{\mu}^{(\phi\circ F) \cdot F^*A }(\sigma)= \mu_\sigma (F^* \phi) \cdot (F^*A )(\sigma)$,
$$
\Xi^{F^{*}(\phi \cdot A)}(\sigma)= (\phi\cdot A)(F_*\sigma),
$$
and that
$$
\Xi_{\mu}^{\phi \cdot A} (F_*\sigma) =\mu_\sigma (F^* \phi)  A(F_*\sigma)= \mu_\sigma (F^* \phi)  \Xi^{ F^*A}(\sigma)\ .
$$
Therefore
\begin{align*}
|\Xi^{F^{*}(\phi \cdot A)} (\sigma)- \Xi_{\mu}^{(F^* \phi) \cdot F^*A }(\sigma)| 
&\leq |\Xi^{F^{*}(\phi \cdot A)} (\sigma)- \Xi_{\mu}^{\phi \cdot A} (F_*\sigma)| +|\Xi_{\mu}^{\phi \cdot A} (F_*\sigma)- \Xi_{\mu}^{(F^* \phi) \cdot F^*A }(\sigma)| \\
&= |(\phi \cdot A) (F_{*}\sigma)- \Xi_{\mu}^{\phi \cdot A} (F_*\sigma)| +|\mu_\sigma (F^* \phi) \big( \Xi^{ F^*A}- F^*A\big)(\sigma)| \\
%&\lesssim \diam(\sigma)^\gamma,
\end{align*}
and the claim follows by \eqref{eq:local increment2} together with Theorem~\ref{thm:multiplication}.
%$$|\Xi_{\mu}^{F^{*}(\phi \cdot A)} (\sigma)- \Xi_{\mu}^{(F^* \phi) \cdot F^*A }(\sigma)| \leq \|F^* \phi\|_{L^\infty} |\Xi_{\mu}^{ F^*A}(\sigma) -(F^*A )(\sigma)|\lesssim \|F\|_{L^\infty} \llbracket \delta \Xi^{F^*A} \rrbracket_\gamma \diam(\sigma)^\gamma\;.   $$
\end{proof}

\section{Embeddings into distribution spaces}\label{sec:embedding}
For $1\leq k\leq d$, we write $\mathfrak{C}^d_{k}:= \{J\subset \{1,...,d\}\ : \ |J|=k \}$.
Given $J\in \mathfrak{C}^d_{k}$ we also write $J^{c}= \{1,..,d\}\setminus J$ and denote by
 $E^J$ the hyperplane spanned by $(e_j : j\in J)$ equipped with its canonical orientation.  
Given  $v,w\in \mathbb{R}^d$ we write $\llbracket w, v \rrbracket :=[w_1,v_1]\times...\times [w_d,v_d]$. 

Given $A\in \Omega^k$, $J\in \mathfrak{C}^d_{k}$ and $\psi \in C^\infty_c(\mathbb{R}^d)$ we define
\begin{equ}\label{eq:embedding}
 \langle\pi_J A, \psi \rangle := (-1)^{|J|} \int_{E^{J^c}} \int_{E^J} A(\llbracket w,w+v \rrbracket) D^{J} \psi(v+w) dv dw\ .
 \end{equ} 
The following is a higher dimensional analogue of \cite[Prop.~3.21]{Che18}.
\begin{prop}\label{prop:embedding}
In the setting above, the map characterised by \eqref{eq:embedding} restricts to a bounded linear map
$\pi_J: \Omega^k_{\alpha, 0}  \to C^{\alpha-1}$. Furthermore the map 
$$\pi  :  \Omega^k_{\alpha, 0} \to \left( C^{\alpha-1}\right)^{\mathfrak{C}^d_k}, \qquad A \mapsto \pi(A)= \big( \pi_J(A) : J \in \mathfrak{C}^d_k \big)$$ 
is injective.
\end{prop}
Recall that $\mass_0(\sigma)=1$ whenever $\sigma\neq 0$ and thus $\Omega^k_{\alpha, 0}= \{A \in \Omega^k \ : \  \|A\|_{\alpha}<+\infty\},$
\begin{remark}\label{rem:notation}
Note that this suggests notation $A= \sum_{J \in \mathfrak{C}^d_{k}} (\pi_J A)(x) dx^J$, which is clearly an identity if $A$ arose from a smooth $k$-form. 
%\blue{consistent with smooth functions, $dg$, pointwise multiplication, pullback, wedge product?}
\end{remark}
\begin{proof}
Noting that for any simplex $\sigma\in \mathfrak{X}^k$ one has
$\int_{E^{J^c}} \int_{E^J} A(\sigma) D^{J} \psi(v+w) dv dw =0 $ and thus 
$$  \int_{E^{J^c}} \int_{E^J} A(\llbracket w,w+v \rrbracket) D^{J} \psi(v+w) dv dw = \int_{E^{J^c}} \int_{E^J} A(\llbracket w +x,w+x+v \rrbracket) D^{J} \psi(v+w) dv dw\ .
$$
Let $\psi\in C^{k+1}_{c}(B_{1/2}(0))$ satisfy $|\psi|_{C^{k+1}}\leq 1$ and write $\psi^{\lambda}_{x}:= \lambda^{-d}\psi\big( (\cdot- x)/\lambda \big)$, then
$$ |\langle\pi_{J} A, \psi^\lambda_x \rangle| \lesssim \lambda^{-k} \sup_{v\in E^J, w\in E^{J^c}, v+w \in \supp (\psi^{\lambda}_{x})  } |A(\llbracket w +x,w+x+v \rrbracket)| \lesssim \| A\|_{\alpha} \lambda^{-k} \lambda^{k-1+\alpha}\ .$$
It remains to check the map is injective. 
Indeed if $\pi_{J} A=0$ for all $J \in  \mathfrak{C}^d_{k}$, this implies by the fundamental Lemma of calculus of variations that for almost every $w$ the function $E^J\ni v\mapsto A(\llbracket w,w+v\rrbracket)\in \mathbb{R}$ is the constant function. Since $A[w,w]=0$ this concludes the proof.
\end{proof}
%
%\gray{
%Introduce for $v\in \mathbb{R}^d$ a translation map $T_v: \mathfrak{X}^k\to \mathfrak{X}^k$ and write
%$$\phi^{\lambda}*A (\sigma)= \int_{\mathbb{R}^d} \phi^{\lambda}(v) A(T_v(\sigma)) dv\ .$$
%Now cite Whitney to say that $\phi^{\lambda}*A$ is just a regular smooth differential form?
%}

\begin{prop}\label{prop:isomorphism}
In the special case $k=d$ the map $\pi$ is surjective with bounded inverse.
\end{prop}
\begin{proof}
Fix an orientation $o$ on $\mathbb{R}^d$.
Given any oriented $d$-simplex $\sigma=(\pmb{\sigma},o_{\sigma})$ we define $\mathrm{sign}(\sigma) = 1$ if $o_{\sigma} = o$ and $\mathrm{sign}(\sigma) = - 1$ otherwise.  
We also set $\mathbf{1}_{\sigma} := \mathrm{sign}(\sigma) \mathbf{1}_{\pmb{\sigma}}$
Furthermore fix $\phi$ as in Remark~\ref{rem: partition of unity} and denote by $\big(\phi^{\simplex}_{n,i}\big)_{n\in \mathbb{N}, i\in I_n}$ the partition of unity constructed from $\phi\in C^\infty$ subordinate to a Whitney decomposition of $\simplex$.

Then, for $F\in C^{\alpha-1}$ one can set\footnote{One can check that the definition is independent of the specific choice of $\phi$.}
 $$A_F(\sigma):= F(  \mathbf{1}_\sigma): =  \mathrm{sign}(\sigma)  \sum_{n\in \mathbb{N},i\in I_n} F(\phi^{\simplex}_{n,i})\;.$$ Let
$\eta=d-1+\alpha$, one then finds that
 \begin{align*}
\sum_{n\in \mathbb{N},i\in I_n} |F(\phi^{\simplex}_{n,i})| &\lesssim \sum_{n\in \mathbb{N},i\in I_n}  \|F\|_{C^{\alpha-1}} 2^{-nd}  2^{-n(\alpha-1)}\lesssim \|F\|_{C^{\alpha-1}} \sum_{n\in \mathbb{N}: 2^{-n}<\diam(\sigma) } |I_n|  2^{-n\eta } \\
&\lesssim \|F\|_{C^{\alpha-1}}  \diam(\sigma)^{\eta},
\end{align*}
where in the last step one argues ad verbatim like in \eqref{eqs:estimates_whitney}. Thus, the following linear map is bounded
$$\iota: C^{\alpha-1}\to \Omega_{\alpha,0}, \qquad F\mapsto A_F\ .$$
When $F$ is smooth, it is easily checked that
$$\big( \pi \circ \iota \big) (F) (w)= \pi A_F (w) = D^{J}_{v} F(\mathbf{1}_{\llbracket w, w+v \rrbracket} )= F(w)\ .$$
Thus, by continuity of $\iota$, it is a right inverse to $\pi$ which implies that $\pi$ is surjective.

%\gray{Next note that $\iota$ is injective, by observing that $ |(F,\phi) |= |\phi \cdot A_F(Q)| \leq \| \phi\| \| A_F\|$.}

\end{proof}

%\subsection{Special case $k=d$}\harprit{See section 2.7 in Chodosh notes.}
%In this section we observe that for $\alpha\in (0,1)$ that there is a canonical isomorphism 
%$$\Omega_{\alpha}^d\sim C^{\alpha-1}\ .$$
%For this it suffices to check that the above injection has a bounded inverse.
%
%
%\gray{
%
%Fixing an orientation of $\mathbb{R}^d$, we identify orentations of $k$ simplices with $-1$ or $1$. Then we can define the injection
%$$\iota: \mathfrak{X}^k \to L^1 (\mathbb{R}^d), \qquad  \sigma= (\pmb \sigma, o) \mapsto (-1)^o \mathbf{1}_{\sigma} ,$$
%This suggests that $ F^{A}\cdot dx:= A \circ \iota$ defines an element of $\Omega_\alpha^{k}$ and $\alpha\in (0,1)$ .
%Furthermore, I guess this might even be an isomorphism.
%}

\section{Application to random fields}\label{sec:application to random fields}
\subsection{A Kolmogorov criterion}
\begin{prop}\label{prop:kolmogorov}
Fix a compact set $\mathfrak{K} \subset \mathbb{R}^{d}$. Let  $A=\big( A_{\sigma}: \sigma \in \mathfrak{X}^{k}(\mathfrak{K}) \big)$ be a $\mathfrak{X}^{k}(\mathfrak{K})$-indexed stochastic process such that the map $S\ni \sigma \mapsto A(\sigma) = A_{\sigma}$ satisfies the relations in Remark~\ref{rem:relation} for any finite subset $S\subset\mathfrak{X}^{k}(\mathfrak{K})$ almost surely. 
%Fix a compact set $\mathfrak{K} \subset \mathbb{R}^{d}$. Let  $A$ be a random element of $\Omega_{k}(\mathfrak{K})$, in particular $\big( A_{\sigma}: \sigma \in \mathfrak{X}^{k}(\mathfrak{K}) \big)$ is a $\mathfrak{X}^{k}(\mathfrak{K})$-indexed stochastic process where the map $\sigma \mapsto A(\sigma) = A_{\sigma}$ satisfies the relations in Definition~\ref{def:kform}. 
Assume furthermore that for some $\tilde{\alpha}, \tilde{\beta} \in (0,1]$ and $q \in \mathbb{N}$,
\begin{equ}\label{eq:kolmogorov_assumption}
M_{q} = \sup_{\osimplex\in \mathfrak{X}_{\le 1}^k(\mathfrak{K}) } \frac{ \EE[|A(\osimplex)|^q]}{ \mass^{k}_{\tilde{\alpha}}(\osimplex)^{q} } +\sup_{\osimplex\in \mathfrak{X}_{\le 1}^{k+1}(\mathfrak{K})} \frac{\EE[|A(\partial\osimplex)|^q]}{ \mass^{k+1}_{\tilde{\beta}}(\osimplex)^{q} } < +\infty\;
\end{equ}
Then, for any $(\alpha,\beta) \in (0,  \tilde{\alpha} \wedge \tilde{\beta} -  d(k+1)/q) \times (0,  \tilde{\beta} -  d(k+2)/q)$, there exists a modification $\hat{A} \in \Omega^{k}_{\alpha,\beta} (\mathfrak{K})$ with
\begin{equ}\label{eq:moment_bound}
\EE[ \|\hat{A}\|_{\alpha,\beta ; \mathfrak{K}}^{q} ] \le M_{q}\;. 
\end{equ}
\end{prop}

As is standard, outside of the proof of Proposition~\ref{prop:kolmogorov} we do not distinguish between $A$ and its modification in our notation, and just write $A$ instead of $\hat{A}$.

\begin{remark}\label{rem:kolmogorov_alt}
Note that the conclusion of Proposition~\ref{prop:kolmogorov} remains true when \eqref{eq:kolmogorov_assumption} is replaced by 
\begin{equ}\label{eq:kolmogorov_assumption2}
M^{\diam}_{q} = \sup_{\osimplex\in \mathfrak{X}_{\le 1}^k(\mathfrak{K}) } \frac{ \EE[|A(\osimplex)|^q]}{ \diam(\sigma)^{q(k-1+\tilde{\alpha})} } +\sup_{\osimplex\in \mathfrak{X}_{\le 1}^{k+1}(\mathfrak{K})} \frac{\EE[|A(\partial\osimplex)|^q]}{ \diam(\sigma)^{q(k+\tilde{\beta})} } < +\infty\ .
\end{equ}
Additionally, the statement also holds if, for $j \in \{0,1\}$, one replaces the suprema over $\mathfrak{X}_{\le 1}^{k+j}(\mathfrak{K})$ in \eqref{eq:kolmogorov_assumption2} with suprema over $\mathfrak{Q}_{\le 1}^{k+j}(\mathfrak{K})$.
%\[
%\mathfrak{Q}_{\le 1}^{k+j} = \{
%Q \in \mathcal{X}_{\le 1}^{k+j}:\ Q \text{ is an oriented }k+j\text{-cube with} \diam(\pmb{Q}) \le 1 \text{ and } \pmb{Q} \subset \mathfrak{K}\}\;.
%\]
\end{remark}
\begin{proof}
We write $(\tilde{\alpha}_{k}, \tilde{\alpha}_{k+1}) = (\tilde{\alpha}, \tilde{\beta})$, $(\alpha_k,\alpha_{k+1}) = (\alpha, \beta)$.
For  $n \in \mathbb{N}$ and $j \in \{k,k+1\}$ we write 
\[
\mathfrak{Y}^{j}_{n}
=
\big\{ \sigma \in \mathfrak{X}_{\le 1}^j(\mathfrak{K}): \mass_{\alpha_j}^{j}(\sigma) \in (2^{-(n-1)\alpha_j},2^{-n \alpha_j}] \big\}\;.
\]
Note that $( \mathfrak{Y}^{j}_{n}: n \in \mathbb{N})$ is a partition of $\mathfrak{X}_{\le 1}^j(\mathfrak{K})$.
For any $\sigma \in  \mathfrak{X}_{\le 1}^{j}(\mathfrak{K})$, we define $n_{\sigma}$ so that $\sigma \in \mathfrak{Y}^{j}_{n_\sigma}$. 
We shall show that
\begin{equs}\label{eq:continuity estimate}
\EE
\Big[ 
\sum_{m=0}^{\infty} 2^{mq \alpha_k }
\sup_{\sigma \in \mathfrak{Y}^{k}_{m}} 
\Big(
\big|A(\pi_{m}\sigma)|^{q} + \sum_{n>m} |A(\pi_{n}\sigma)-A(\pi_{n-1}\sigma)|^{q}
\Big)
\Big] &< \infty\ ,
\end{equs}
where the maps $\{\pi_n\}_{n\in \mathbb{N}}$ were introduced at the end of Section~\ref{sec:distributional}.
This will imply that, with probability $1$, for any $\sigma \in \mathfrak{X}_{\le 1}^{k}(\mathfrak{K})$, the limit 
$
\hat{A}(\sigma) := 
\lim_{n \rightarrow \infty} 
A(\pi_{n}\sigma)\
$
exists
%Moreover, by additivity of $A$, the same holds for  $\sigma \in \mathfrak{X}^{k}(\mathfrak{K})$ 
and one can extend $\hat{A}$ to $\mathcal{X}^k(\mathfrak{K})$.
Since measurability is preserved in this limit, we have $\hat{A} \in \Omega^{k}$. 
In particular, for $\sigma \in \mathfrak{X}^{k+1}(\mathfrak{K})$, $\lim_{n \rightarrow \infty} A(\pi_{n}\partial\sigma) = \hat{A}(\partial \sigma )$, which together with showing 
\begin{equ}\label{eq:continuity estimate2}
\EE
\Big[ 
\sum_{m=0}^{\infty} 2^{mq \alpha_{k+1} }
\sup_{\sigma \in \mathfrak{Y}^{k+1}_{m}} 
\Big(
\big|A(\pi_{m} \partial \sigma)|^{q} + \sum_{n>m} |A(\pi_{n} \partial \sigma)-A(\pi_{n-1} \partial \sigma)|^{q}
\Big)
\big] < \infty\;,
\end{equ}
will prove that $\hat{A}$ satisfies the estimate \eqref{eq:moment_bound}. 

We use the parameter $j \in \{k,k+1\}$ so we can write discuss estimates \eqref{eq:continuity estimate} and \eqref{eq:continuity estimate2} at the same time. 
We claim that, uniformly in $m$, $\sigma \in \mathfrak{Y}^{j}_{m}$ and $n \ge m$, we have 
\begin{equs}\label{equ:single_term_estimate}
\EE [ |A(\pi_{m}\sigma_j)|^{q} ] 
&\lesssim 
M_{q}
2^{-q m \hat{\alpha}_j} 
\quad
\text{and}
\quad
\EE
[
|A(\pi_{n}\sigma_j)-A(\pi_{n-1}\sigma_j)|^{q}]
&\lesssim 
M_{q}
 2^{-nq \hat{\alpha}_j}   \;,
\end{equs}
where, $\sigma_k = \sigma$, $\hat{\alpha}_{k} =  \tilde{\alpha} \wedge \tilde{\beta}$, and $\sigma_{k+1} = \partial \sigma$, $\hat{\alpha}_{k+1} = \tilde{\beta}$. 

%To prove the inequality \eqref{equ:single_term_estimate} we argue using Corollary~\ref{cor:dyadic_diff_est}.
{We turn to showing inequality \eqref{equ:single_term_estimate}.}
For any $\sigma  \in \mathfrak{X}_{\le 1}^{j}$ and $n,n' \in \mathbb{Z}$, there exists, {by Corollary~\ref{cor:dyadic_diff_est}}, a $Z^j = Z^j(\sigma,n,n') \in \mathcal{X}^{j+1}(\mathfrak{K})$ such that, uniform in such $\sigma,n,n'$, 
\[
\Big| 
 \pi_{n}(\sigma) - \pi_{n'}(\sigma) -  \partial Z^j 
 \Big|_{\tilde{\alpha}_j;j}
\lesssim 
2^{-(n \wedge n') \tilde{\alpha}_j}\;,
\]
and, for $j=k$, we also have 
\[
\Big| Z^k \Big|_{\tilde{\beta};{k+1}} \lesssim 2^{-(n \wedge n') \tilde{\beta}}\;.
\]
Continuing in the case $j=k$, we have
\begin{equs}\label{moment_estimates}
\EE\big[
\big|A ( 
\pi_{n}(\sigma) - \pi_{n'}(\sigma)  
)
\big|^{q} \big] 
&\lesssim 
\EE\big[ \big|A ( \pi_{n}(\sigma) - \pi_{n'}(\sigma) -  \partial  Z^{k})\big|^{q}]
+
\EE\big[ \big|A (\partial Z^{k})\big|^{q} \big]\\
& \lesssim M_{q} \Big( \big|
 \pi_{n}(\sigma) - \pi_{n'}(\sigma) -  \partial Z^{k} \big|_{\tilde{\alpha};{k}}^q + \big|Z^k \big|^{q}_{\tilde{\beta};{k+1} } \Big)\\
&\lesssim 
M_{q} \big( 2^{-\tilde{\alpha}(n \wedge n' )q} +  2^{-\tilde{\beta}(n \wedge n' )q} \big) = M_{q}
 2^{- (n \wedge n' ) q \hat{\alpha}_k}  \;.
\end{equs}
%where in the second inequality of \eqref{moment_estimates}, we used that for any decomposition  $ \pi_{n}(\sigma) - \pi_{n'}(\sigma) -  \partial Z = \sum_{i} c_i \omega_i$ with $Z = \sum_{i} \bar{c}_{i} \bar{\omega}_i$,  
%we have
%\begin{equs}
%\EE\big[ \big|A ( \pi_{n}(\sigma) - \pi_{n'}(\sigma) -  \partial Z )\big|^{q} \big]^{1/q}
%&\leq
% \sum_{i} |c_i| \EE\big[ \big|A (\omega_i)\big|^{q}\big]^{1/q}
%\leq
%M_{q}^{1/q} \sum_{i} |c_i| \mass_{\tilde{\alpha}}^{k}(\omega_i)\\
%\EE\big[ \big|A (\partial Z^{k}) \big|^{q} \big]^{1/q}
%&\leq
% \sum_{i} |\bar{c}_i| \EE\big[ \big|A ( \partial  \bar{\omega}_i)\big|^{q} \big]^{1/q}
% \le 
%M_{q}^{1/q}\sum_{i} |\bar{c}_i| \mass_{\tilde{\beta}}^{k+1}(\bar{\omega}_i)\
%\end{equs}
%By taking the infimum over decompositions we get the we can replace the sums on the right hand sides above by $\mass_{\tilde{\alpha}}^{k}
%\big( \pi_{n}(\sigma) - \pi_{n'}(\sigma) -  \partial Z^{k} \big)$ and $\mass_{\tilde{\beta}}^{k+1}(Z^k)$ which gives the second inequality in \eqref{moment_estimates}. 

For $j=k+1$, since $\partial \partial Z^{k+1} = 0$, we have
\begin{equs}
{}&\EE\big[ 
\big|
A
\big( \partial \pi_{n}(\sigma) - \partial \pi_{n'}(\sigma)\big)\big|^{q}
\big]
=
\EE\big[
\Big|
A
\Big( \partial \big(\pi_{n}(\sigma) -  \pi_{n'}(\sigma) - \partial Z^{k+1}  \big)\Big)
\Big|^{q}
\big]\\
{}&\lesssim 
M_{q} \Big|  \pi_{n}(\sigma) -  \pi_{n'}(\sigma) - \partial Z^{k+1} \Big|_{\tilde{\beta};k+1}^q
\lesssim 
M_{q} 2^{-\tilde{\beta}(n \wedge n')q} = M_{q} 2^{-\hat{\alpha}_{k} (n \wedge n')q} \;.
\end{equs}

Now that \eqref{equ:single_term_estimate} is proved, let $\mathfrak{Y}^{j}_{m,n} = \big\{\pi_{n} \sigma: \sigma \in \mathfrak{Y}^{j}_{m} \big\}$.
Observe that  
\[
|\mathfrak{Y}^{j}_{m,n}| \le \diam(\mathfrak{K})^{(j+1)d} 2^{(j+1)dn}]\;,
\] 
since each such simplex is determined by $j+1$ vertices in $\mathfrak{K} \cap (2^{-n} \mathbb{Z}^{d})$. 
We then have 
\begin{equs}
 \EE &
\Big[ 
\sum_{m = 0}^{\infty} 2^{mq \alpha_j}
\sup_{\sigma \in \mathfrak{Y}^{j}_{m}} 
\Big(
\big|A(\pi_{m}\sigma_j)|^{q} + \sum_{n>m} |A(\pi_{n}\sigma_j)-A(\pi_{n-1}\sigma_j)|^{q}
\Big)
\Big]\\
{}& \lesssim
 \EE
\Big[ 
\sum_{m = 0}^{\infty} 2^{mq \alpha_j}
\Big(
\sum_{\tilde{\sigma} \in \mathfrak{Y}^{k}_{m,m}}
\big|A(\tilde{\sigma}_j)|^{q} + \sum_{n>m} \sum_{\tilde{\sigma} \in \mathfrak{Y}^{k}_{m,n}}
|A(\tilde{\sigma}_j)-A(\pi_{n-1} \tilde{\sigma}_j)|^{q}
\Big)
\Big]\\
{}&\lesssim 
M_{q}
\sum_{m = 0}^{\infty}
2^{mq \alpha_j}
\Big(
 2^{-q m\hat{\alpha}_{j} } 2^{(j+1)dm}
+
 \sum_{n>m} 2^{-nq \hat{\alpha}_j } 2^{(j+1)dn} 
 \Big) 
 \\
{}& \lesssim M_{q} \sum_{m = 0}^{\infty}  2^{q m ( \alpha_j -  \hat{\alpha}_j) + (j+1)dm} \lesssim M_{q}\;,
\end{equs}
where in  the third inequality we used that $\tilde{\alpha}_j  > (j+1)d/q$ and in the final inequality we used that $\alpha_j < \hat{\alpha}_j - d(j+1)/q$. 
\end{proof}

\subsection{Rough Gaussian $k$-forms}
In this section we exhibit a criterion for fractional Gaussian fields, cf. \cite{LSSW16}, to belong to the spaces $\Omega_{\alpha,\beta}^k$.
%\gray{
%We consider for positive definite kernels $K^{I,J}: \mathbb{R}^d\setminus \{0\}\to \mathbb{R}$ the space $\mathcal{M}_K$ of signed Borel measures $\mu= \{\mu^{I}\}$, such that 
%$$\int |K^{I,J}(x,y)| d|\mu^I|(x)d|\mu^J|(y) <+\infty.$$ Then, consider centred Gaussian fields $\{X_I\}_I$ over $\mathcal{M}_K$ such that
%$$E[X_I(\mu) X_I(\nu)]= \int K(x,y) d\mu(x)d\nu(y)\ $$ for all $\mu, \nu\in \mathcal{M}_K$.
%If $\mathcal{M}_K$ contains the Hausdorff $V_\sigma$ measure of simplices $\sigma\in \mathfrak{X}^k$ characterised by $f\mapsto \int_{\mathbb{R}^{d}} f(x) \mathbf{1}_{\pmb{\sigma}}(x) d\Vol^{k}(x)$, then we can define 
%$ A= \sum_{I} X_I dx^I $ by
%$$A(\sigma)= \sum_{|I|=k} X_I(\sigma) \frac{dx^I(\sigma)}{|dx^I (\sigma)|} \in \Omega^k .$$
%} \ajay{Tempted to formulate things at the level of smooth objects}
%In this subsection we apply the Kolmogorov criterion from earlier to give a more straightforward criterion Proposition~\ref{prop:gaussian} that can be used in the Gaussian case. %
%\begin{theorem}
%Log corrolated fields and GFFs as example.
%\end{theorem}
We start with the following generalisation of \cite[Lemma~4.9]{CCHS22}. 
\begin{lemma}\label{lemma:cube_estimate}
Let $\pmb{Q} \subset \mathbb{R}^{d}$ be a $k$-cube of side length $r \in (0,1]$. 
We associate to $\pmb{Q}$ the distribution $\delta_{\pmb{Q}}$ on $\mathbb{R}^{d}$ by setting, for any smooth $f$ on $\mathbb{R}^{d}$,  
\[
\langle \delta_{\pmb{Q}},f\rangle = \int_{[0,r]^{k}} f(\gamma(x)) \mathrm{d}^{k}x\;,
\]
where $\gamma: [0,r]^{k} \rightarrow  \pmb{Q}$ is an isometric embedding and $\mathrm{d}^{k}x$ the Lebesgue measure. There exists $C>0$ independent of $r$ such that
$
\| \delta_{\pmb{Q}} \|_{H^{-\theta}(\mathbb{R}^{d})} 
\leq C r^{ (k + (2 \theta - d + k) \wedge k)/2}
$
for any $\theta > (d-k)/2$.
\end{lemma}
\begin{proof}
Without loss of generality, we take $\pmb{Q} = [0,r]^{k} \times \{0\}^{d-k} \subset \mathbb{R}^{d}$. 
We then have
\[
\widehat{\delta_{\pmb{Q}}}(p_1,\dots,p_{n})
=
\prod_{j=1}^{k} \frac{e^{2\pi i p_j r} - 1}{2 \pi i p_{j}}\;.
\]
Thus,
\begin{equs}
\| \delta_{\pmb{Q}} \|_{H^{-\theta}(\mathbb{R}^{d})}^2 
&=
\int_{\mathbb{R}^{d}} \big(1 + \sum_{j=1}^{d} |p_{j}| \big)^{-2 \theta} \big| \widehat{\delta_{\pmb{Q}}}(p) \big|^2 \; \mathrm{d}^{d}p
\lesssim
\int_{\mathbb{R}^{d}} \big(1 + \sum_{j=1}^{d} |p_{j}| \big)^{-2 \theta} \prod_{j=1}^{k}( r^2 \wedge p_{j}^{-2})\;  \mathrm{d}^{d}p \\
&
\lesssim
\int_{\mathbb{R}^{k}} \big(1 + \sum_{j=1}^{k} |p_{j}| \big)^{(d-k)-2 \theta} \prod_{j=1}^{k}( r^2 \wedge p_{j}^{-2})\;  \mathrm{d}^{k}p
\;.
\end{equs}
The estimate then follows from the observation that for any $0 \le m \le k$, 
\begin{equs}
\int_{\mathbb{R}^{k}} \big(1 + \sum_{j=1}^{k} |p_{j}| \big)^{(d-k)-2 \theta} \prod_{j=1}^{m} \big( p_{j}^{-2} \mathbf{1} \big\{ |p_{j}| \ge r^{-1} \big\} \big)
\prod_{j=m+1}^{k} \big( r^{2}  \mathbf{1} \big\{ |p_{j}| < r^{-1} \big\}  \big) \mathrm{d}^{k}p
\lesssim r^{k +  (2\theta - d + k) \wedge k} \;.
\end{equs}
\end{proof}

\begin{prop}\label{prop:gaussian-kform}
Fix $k < d$, $(d-k)/2 < \theta $.
Let $(A_{I}: I\in \mathfrak{C}^d_{k})$ be a collection of centred, jointly Gaussian, random fields on $H^{-\theta}(\mathbb{R}^{d})$ for which there exists $C>0$ such that
\[
\EE[ A_{I}(f)^2] \le C \|f\|_{H^{ - \theta}}^2\qquad \text{uniformly over $I$ and }f \in H^{-\theta}(\mathbb{R}^{d})\;.
\] 
Let $\bar{\alpha}= (\theta-d/2+1)\wedge 1$ and $\bar \beta = (\theta-d/2) \wedge 1$. Then, there exists
$A\in \Omega^{k}_{\alpha,\beta}(\mathfrak{K}) $ 
such that $\pi(A)$ is a modification of $(A_{I}: I\in \mathfrak{C}^d_{k})$ and
$
\EE \big[ \llbracket A\rrbracket_{(\alpha,\beta); \mathfrak{K}}^2 \big]
\lesssim_{\mathfrak{K}} C\;
$
for every $\alpha\in (0,\bar{\alpha}]$, $\beta\in (0, \bar{\beta}]$ and compact $\mathfrak{K} \subset \mathbb{R}^{d}$.
%
%compact set $\mathfrak{K} \subset \mathbb{R}^{d}$, 
%
%Writing \gray{$\zeta = \theta - d/2 - 1/2$, it then follows that, for any $0 < \alpha,\beta < \zeta$} and compact set $\mathfrak{K} \subset \mathbb{R}^{d}$, 
%there exists $A\in \Omega^{k}$ satisfying
%\[
%\EE \big[ \|A\|_{(\alpha,\beta); \mathfrak{K}}^2 \big]
%\lesssim C\;. 
%\]
%and $\pi(A)= (A_{I}: I\in \mathfrak{C}^d_{k})$.
%where $A \in \Omega^{k}$ is defined by the classical smooth $k$-form $A = \sum_{I} A_{I} \mathrm{d}x^{I}$ on $\mathbb{R}^{d}$. 
\end{prop}
\begin{proof}
By Proposition~\ref{prop:kolmogorov} and Remark~\ref{rem:kolmogorov_alt} with the choice $\bar{\alpha} = \theta/2 - d/2 + 1$ and $\bar{\beta} = \zeta$ and combined with equivalence of moments for Gaussian random variables, it suffices to prove the estimate
\begin{equ}\label{eq:example_goal}
 \sup_{ Q \in \mathfrak{Q}_{\le 1}^k(\mathfrak{K}) } \frac{ \EE[|A(Q)|^2]}{ \diam(Q)^{2(k-1 + \bar \alpha)} } +\sup_{Q \in \mathfrak{Q}_{\le 1}^{k+1}(\mathfrak{K})} \frac{\EE[|A(\partial Q)|^2]}{ \diam(Q)^{2(k +\bar \beta)} } 
 \lesssim C\;.
\end{equ}
For $Q \in \mathfrak{Q}^{k}(\mathfrak{K})$ of side length $r$, let
$
A(Q) = \frac{1}{r^k} \sum_{I} \mathrm{d}x^{I}(Q) \delta_{\pmb{Q}}(A_{I})\;.
$
To estimate the first term on the left-hand side of \eqref{eq:example_goal} we note that by Lemma~\ref{lemma:cube_estimate} we have, uniform in $Q \in  \mathfrak{Q}_{\le 1}^{k}(\mathfrak{K})$, 
\[
\EE[
\delta_{\pmb{Q}}(A_{I})^2]
\le C \| \delta_{\pmb{Q}} \|_{H^{-\theta}}^2
\lesssim C \diam(Q)^{k + (2 \theta - d + k)\wedge k }= C \diam(Q)^{2(k  + \theta - d/2) \wedge 2k}\;.
\]
We turn to estimating the second term of \eqref{eq:example_goal}. 
The desired estimate follows from applying Stokes' theorem estimate to write  $|A(\partial Q)| = |(\partial A)(Q)|$ and then observing that, for any $I$ and $j \in [d] \setminus I$, and uniform in  $Q \in  \mathfrak{Q}_{\le 1}^{k+1}(\mathfrak{K})$,  
\[
\mathbb{E}[\delta_{\pmb{Q}}(\partial_{j} A_{I})^2]
\lesssim C\| \partial_{j} \delta_{\pmb{Q}}\|_{H^{-\theta}}^2
\lesssim C\| \delta_{\pmb{Q}}\|_{H^{-\theta + 1}}^2
\lesssim C \diam(Q)^{k + 1 + (2 (\theta - 1) - d + k+1)\wedge(k+1) }   \;.
\]
\end{proof}

\section{Methods of subdivision}\label{sec:sewing}
%$$\triangle_k=\left(\left\{\sum_{i=1}^k t_i e_i \in \mathbb{R}^{k} : \ t_i\in [0,1],\  \sum_{i=1}^k t_i\leq 1 \right\}, \ (0,e_1,...,e_k)   \right) \ .$$
%\harprit{
%Not clear to me how to obtain the analogue of $|\mathcal{I} A|(\sigma)\lesssim \mass_\alpha^k (\sigma)$ for more general controls.
%}
%\gray{We call a family of subdivisions $\{\mcK_i\}_{i\in I}$ regular if 
%$\sup_{i\in I}|\mcK_i|\diam(\mcK_i)^k<+\infty$.}
%\end{definition}
%\begin{remark}
%A subdivision is not required to be a simplicial complex itself. 
%\end{remark} \ajay{Cut this remark if we don't need simplicial complexes}

We define the eccentricity $\mathfrak{e}(\sigma)$ of $\sigma \in \mathfrak{X}^k$ as
\begin{equ}\label{eq:eccentricity}
\mathfrak{e}(\sigma) 
:=
\frac{\diam(\sigma)^k}{\Vol^k(\sigma)}\;.
\end{equ}
\begin{definition}\label{def:reg_method}
We call a family of maps $\mathcal{M}= \{\mathcal{M}_{\ell}\}_{\ell\in \mathbb{N}}$ where for each $\ell \in \mathbb{N}$, 
$$\mathcal{M}_{\ell}: \mathfrak{X}^k \mapsto 2^{\mathfrak{X}^k}, $$ a method of subdivision (for $k$ simplices), if for each $\sigma \in \mathfrak{X}^k$, $\mathcal{M}_\ell(\sigma)$ is a subdivision of $\sigma$ and
\begin{equ}\label{eq:semigroup}
\mathcal{M}_{\ell+\ell'}(\sigma)= \bigcup_{w\in \mathcal{M}_{\ell'}(\sigma)} \mathcal{M}_{\ell}(w)\ , \qquad \text{for every }\ell,\ell'\in \NN\;. 
\end{equ}
We call $\mathrm{card}(\mathcal{M}) := \sup_{\sigma}|\mathcal{M}_1(\sigma)| $ the cardinality of $\mathcal{M}$ and write
\[
\| \mathcal{M} \|
:=
\sup_{\sigma \in \mathfrak{X}^k}
\sup_{\ell \in \mathbb{N}}
\sup_{w \in\mathcal{M}_\ell(\sigma) } 
\frac{\mathfrak{e}(w)}{\mathfrak{e}(\sigma)}
\quad
\text{and}
\quad
\mathfrak{c}_{\mathcal{M}}
:=
\sup_{\sigma \in \mathfrak{X}^k}
\max_{w \in\mathcal{M}_1(\sigma) } 
\Big( \frac{\diam(w)}{\diam(\sigma)} \Big)\; \] 
as well as 
\begin{equ}
\vertiii{ \mathcal{M}} := \inf_{\mu>0}\sup_{l\in \mathbb{N}} \ell^{-\mu}  \sup \Big\{ 
\frac{\Vol^k(w')}{\Vol^k(w)}: \enskip \sigma \in \mathfrak{X}^{k},\;w,w'\in \mcM_\ell(\sigma) \Big\} \;.
\end{equ}
We call $\mathcal{M}$ \textit{strongly regular} if  $\mathrm{card}(\mathcal{M}) \vee  \| \mathcal{M} \|\vee \vertiii{ \mathcal{M}} < \infty$ and  $\mathfrak{c}_{\mathcal{M}} < 1$.
%\gray{ We call a family of methods of subdivision $\{\mcM^{(k)}\}_{k=1}^d$ \textit{closed} if $$\sum_{\sigma' \in \mcM^{k}(\sigma)} \partial\sigma' = \sum_{\sigma'' \in \mcM^{k}(\partial\sigma)} \sigma'' \ .$$}
\end{definition}
\begin{remark}
Enforcing $\| \mathcal{M} \| < \infty$ rules out methods of iterative subdivision that would be problematic for our analysis, such as slicing a triangle into thinner and thinner strips without significantly reducing the length of the strips.
\end{remark}
\begin{remark}
The main theorem of \cite{EG} states that there exists a strongly regular method of subdivision with $\vertiii{ \mathcal{M}}=1$\ .
\end{remark}

\begin{remark}
Another natural method of subdivision is longest-edge bisection, but the question of whether $\| \mathcal{M} \| < \infty$ for this method appears to be an open problem, see \cite{KOR}.
\end{remark}

\begin{remark}
For any strongly regular method of subdivision $\mathcal{M}$, $\sigma\in \mathfrak{X}^k$, and $n \in \mathbb{N}$, 
$$\sum_{\sigma'\in \mathcal{M}_n(\sigma) }  \diam(\sigma')^{k}  
\leq
% \max_{{\sigma'\in \mathcal{M}_n(\sigma) }} \frac{\diam(\sigma')^k}{\Vol^k(\sigma')} \Vol^k(\sigma)
\Big(
 \max_{{\sigma'\in \mathcal{M}_n(\sigma) }} \mathfrak{e}(\sigma') 
 \Big) \Vol^k(\sigma)
\leq \| \mcM\| \mathfrak{e}(\sigma)\Vol^k(\sigma)
 =  \| \mcM\| \diam(\sigma)^k\;.$$
In particular,  then sequence $\big(\mathcal{M}_{l}(\sigma) \big)_{l\in \mathbb{N}}$ is a regular sequence of subdivisions of $\sigma$ in the sense of Definition~\ref{def:subdivision,mesh} since, for $\gamma>k$, 
\begin{equs}
\sum_{\sigma'\in\mcM_n (\sigma) } \diam(\sigma')^{\gamma} &
 \leq \max_{\sigma'\in \mcM_n(\sigma)}   \diam(\sigma')^{\gamma-k} 
\sum_{\sigma'\in \mathcal{M}_n(\sigma) }  \diam(\sigma')^{k} \\
&\leq \| \mcM\| \diam^{k} (\sigma) \max_{\sigma'\in \mathcal{M}_n(\sigma)}   \diam(\sigma')^{\gamma-k}\\
& \leq  c^{n(\gamma-k)}_\mcM \| \mcM\| \diam^{\gamma} (\sigma)  \ .\label{eq:core_summation bound for regular methods}
\end{equs}
\end{remark}
%
%For two methods of subdivision $\mathcal{M}$, $\mathcal{M}'$, set
Given two subdivisions $\mathcal{K} | \sigma$ and $\mathcal{K}' | \sigma$, we say $\mathcal{K}'$ is a refinement of $\mathcal{K}$ if, for every $\tau \in \mathcal{K}$, there exists $\mathcal{K}'' \subset \mathcal{K}'$ with $\mathcal{K}'' | \tau$. 
\begin{lemma}\label{lemma:common refinement}
Consider two strongly regular methods of subdivision $\mathcal{M},\ \mathcal{M}'$. Then, there exists $\mu= \mu(\mcM)>0$, $C=C(\mcM)$ 
such that, for any $\sigma$ and $n\in \mathbb{N}$, there exists a common refinement $\mathcal{K}_n= \mathcal{K}(\mathcal{M},\mathcal{M}',\sigma,n) $ of both $\mathcal{M}_n(\sigma)$ and $\mathcal{M}'_n(\sigma)$ satisfying 
$$\max_{w'\in \mcM'_n(\sigma)}
\big|
 \{ \tilde{w}\in \mcK_n : \tilde{w}\subset w' \} \big| 
 \lesssim_{k } C(\mcM)  \mathfrak{e}(\sigma) n^{\mu(\mcM)} \left(1+\frac{{\diam \big(\mcM'_n(\sigma) \big)}}{\diam\big(\mathcal{M}_n(\sigma)\big)} \right)^k\ .$$
\end{lemma}
Note that by symmetry one can reverse the roles of $\mcM$ and $\mcM'$ in the estimate above. 
\begin{proof}
Set $$\mathcal{D}_{n;\sigma}^{\mathcal{M} ,\mathcal{M}'}(w') :=   
\big|
\{w\in \mathcal{M}_n(\sigma): \Vol^k(w\cap w') \neq 0\}
\big|. $$
The lemma follows by combining the following claim with the fact that for any two simplices, their intersection can be written as a union of a bounded ($k$-dependent) number of simplices. 
\begin{claim}
There exists $ \mu(\mcM), C(\mcM)>0$
such that 
$$\mathcal{D}_{n;\sigma}^{\mathcal{M} ,\mathcal{M}'}(w')\leq C(\mcM) \mathfrak{e}(\sigma) n^{\mu} \left(1+\frac{{\diam (w')}}{\diam(\mathcal{M}_n(\sigma))} \right)^k\ $$
uniformly in $n\in\mathbb{N}$. 
\end{claim}

%
%
%such that for 
%any $n\in\mathbb{N}$ 
%$$\mathcal{D}_n^{\mathcal{M} ,\mathcal{M}'}(\sigma)\lesssim \mathfrak{e}(\sigma) n^{\mu} \left(1+\frac{{\diam (w')}}{\diam(\mathcal{M}_n(\sigma))} \right)^k\ .$$
We now prove this claim.
Write $\mathcal{K}_{n}(w'):=\{w\in \mathcal{M}_n(\sigma): \Vol^k(w\cap w') \neq 0\}$ and let $\bar w\in \mcM(\sigma)$ be such that $\diam (\mcM_n(\sigma))=\diam(\bar w)$. Then,
\begin{align*}
|\mathcal{K}_{n}(w')| \diam(\bar w)^{k} & \leq  |\mathcal{K}_{n}(w')| \mathfrak{e}(\bar w)  \Vol^k(\bar{w})\leq 
 \mathfrak{e}(\bar w)\vertiii{\mathcal{M}}  n^{\mu(\mcM)} \sum_{w\in \mathcal{K}_{n}(w')} \Vol^k(w)
%
%&\leq 
%\vertiii{\mathcal{M}} \mathfrak{e}(w') n^{\mu(\mcM)}\sum_{w\in \mathcal{K}_{n,w'}}  \Vol(w)\\
%
\\
 & \leq \| \mcM\| \cdot \vertiii{\mathcal{M}} \mathfrak{e}(\sigma) n^{\mu(\mcM)} \sum_{w\in \mathcal{K}_{n}(w')} \Vol(w) \\
 &\leq \| \mcM\|  \cdot \vertiii{\mathcal{M}} \mathfrak{e}(\sigma) n^{\mu(\mcM)} (\diam(\bar{w})+{\diam (w')})^k\ 
\end{align*}
and therefore, 
$|\mathcal{K}_{n}(w')|\leq \| \mcM\|  \vertiii{\mathcal{M}}  \mathfrak{e}(\sigma) n^{\mu(\mcM)} \left(1+\frac{{\diam (w')}}{\diam(\bar w )} \right)^k\ .$
\end{proof}
%\red{
%\begin{remark}\label{rem:flexibility_remark}\harprit{modify to fit newer statement}
%We leave it as an exercise for the reader to note that for $\mathcal{K}$ a subdivision of $\sigma$ such that $\Vol^k(\sigma')>0$ for all $\sigma'\in \mathcal{K}$, there exists a regular method $\mcM$ such that $\mcK= \mcM_1(\sigma)$.
%\end{remark}
%}
\subsection{Proof of Proposition~\ref{prop:sewing}} 
We shall suppress $\mathfrak{K}\subset \mathbb{R}^d$ from the notation throughout the proof.

\begin{lemma}\label{lem:sewing construction}
In the setting of Proposition~\ref{prop:sewing}, for any strongly regular method of subdivision $\mathcal{M}$, the map 
$$\mathcal{I}_{\mathcal{M} } \Xi: \mathfrak{X}^k \to \mathbb{R},\qquad \sigma \mapsto  \mathcal{I}_{\mathcal{M} } \Xi (\sigma):= \lim_{n\to \infty } \sum_{\sigma'\in \mcM_n(\sigma)} \Xi(\sigma')$$
satisfies
 $$ |\mathcal{I}_{\mathcal{M} } \Xi (\sigma)-\Xi (\sigma) | \leq \frac{\| \mcM\| {\mathrm{card}(\mathcal{M})} }{1-\mathfrak{c}^{(\gamma-k)}_\mcM}  \llbracket \delta\Xi\rrbracket_{\gamma,\mathfrak{K}}
 \diam(\sigma)^{\gamma} \ .
 $$ as well as
 $\mathcal{I}_{\mcM } \Xi (\sigma) = \sum_{\sigma'\in \mathcal{M}_n(\sigma) } \mathcal{I}_{\mathcal{M} }\Xi (\sigma')$.
\end{lemma}

\begin{proof}

Using \eqref{eq:core_summation bound for regular methods} in the last line we find that
\begin{align*}
\Bigg|\sum_{\sigma'\in \mcM_{n+1}(\sigma)} \Xi(\sigma')-\sum_{\sigma'\in \mcM_n(\sigma)} \Xi(\sigma') \Bigg| 
&\leq \sum_{\sigma'\in \mcM_n(\sigma)}   \Big(\Xi(\sigma')- \sum_{\sigma''\in \mcM_1(\sigma')}  \Xi(\sigma'')\Big) \\
&\leq \llbracket \delta\Xi\rrbracket_{\gamma,\mathfrak{K}}  {\mathrm{card}(\mathcal{M})}  \sum_{\sigma'\in \mcM_n(\sigma)}   \diam (\sigma')^\gamma \\
& \leq  \mathfrak{c}^{n(\gamma-k)}_\mcM \| \mcM\| {\mathrm{card}(\mathcal{M})} \llbracket \delta\Xi\rrbracket_{\gamma,\mathfrak{K}} \diam (\sigma)^{\gamma} \ .
\end{align*}
Since $c_\mcM<1$ and $\gamma>k$, this is summable in $n$. The remaining claim follows from \eqref{eq:semigroup}.
\end{proof}
%\begin{align*}
%|A(\gamma(\mathcal{K}^{(0)})-\gamma(\mathcal{K}^{(n)}) | &\lesssim C |A|_{\bar{\eta}} \sum_{\sigma\in \mathcal{K}^{(0) }} |\gamma|_{\alpha;\sigma}^{\bar{\alpha}/\alpha} \Vol (\sigma) \cdot 
%\left(\sum_{k=0}^n \max_{\sigma'\in \mathcal{K}^{(n) }_\sigma}   \diam(\sigma')^{k(\bar{\alpha}/\alpha-1)}  \right) \\
%&\lesssim C_{\mathcal{M}} |A|_{\bar{\alpha}} \sum_{\sigma\in \mathcal{K}^{(0) }} |\gamma|_{\alpha;\sigma}^{\bar{\alpha}/\alpha} \Vol (\sigma) 
% \diam(\sigma)^{k(\bar{\alpha}/\alpha-1)} 
%\end{align*}
We define the following equivalence relation between methods of subdivisions: we say that 
$\mcM\sim\mcM'$, whenever there exists $C>0$ such that $\frac{1}{C}\leq \sup_{n}\frac{\diam(\mcM_{n} (\sigma))}{\diam( \mcM'_{n} (\sigma))} \leq C\ $ uniformly over $\sigma\in \mathfrak{X}^k_{\leq 1}$, $n\in \mathbb{N}$.

\begin{lemma}\label{lem:M_independent of equivalence class}
Given two strongly regular methods of subdivision $\mcM$ and $\mcM'$, the maps $\mathcal{I}_{\mathcal{M} }\Xi(\sigma)$ and $\mathcal{I}_{\mcM'}\Xi(\sigma)$ agree whenever $\mcM\sim\mcM'$.
\end{lemma}

\begin{proof}
%Let  strongly regular methods of subdivision $\mcM$, $\mcM'$. 
By Lemma~\ref{lemma:common refinement} there exists $\mu=\mu(\mcM,\mcM')$ and a common refinement $\mathcal{K}_n$ of $\mcM_n(\sigma), \mcM'_n(\sigma)$ such that, if we write $\mathcal{K}_n(w):= \left\{ \tilde{ w} \in \mathcal{K}_n\ : \ \tilde w\subset w \right\} $  for $w\in \mcM_n(\sigma) \cup \mcM'_n(\sigma)$, 
then
$$|\mathcal{K}_n(w)| \lesssim_{\mathcal{M},\tilde {\mathcal{M}},\sigma} n^{\mu}  .$$
Thus
\begin{align*}
\mathcal{I}_{\mathcal{M} } \Xi(\sigma) - \mathcal{I}_{\mathcal{M}' } \Xi (\sigma) &=
 \sum_{w\in  \mathcal{M}_n(\sigma)} \mathcal{I}_{\mathcal{M} } \Xi(w) -  \sum_{w\in \mathcal{M}_n'(\sigma)} \mathcal{I}_{\mathcal{M}' } \Xi(w) \\
&= \sum_{w\in  \mathcal{M}_n(\sigma)} \left( \mathcal{I}_{\mathcal{M} } \Xi(w)- \Xi(w) \right)  - \sum_{w'\in \mathcal{M}_n'(\sigma)} \left( \mathcal{I}_{\mathcal{M}' } \Xi(w')- \Xi(w') \right)\\
& \quad + \sum_{w\in  \mathcal{M}_n(\sigma)} \Big( \Xi(w) -\sum_{\tilde w\in \mathcal{K}_n(w) }  \Xi(\tilde w)   \Big)  
- \sum_{w'\in \mathcal{M}_n'(\sigma)} \Big( \Xi(w') 	-	\sum_{\tilde w\in \mathcal{K}_n(w') } 	\Xi(\tilde w)	\Big)\ .
\end{align*}

The first two sums in the last expression above converge to $0$ as $n\to \infty$ by definition, and for the third sum we have \begin{align*}
\Big| \Xi(w) -\sum_{\tilde w\in \mathcal{K}^{M,\tilde{M}}_n(\sigma); \tilde w\subset w }  \Xi(w)   \Big|  
&= | \delta_{w,\mathcal{K}_n(w)} \Xi| \leq |\mathcal{K}_n(w)| \llbracket\delta\Xi\rrbracket_\gamma \diam(w)^\gamma\\
&\lesssim n^\mu  \llbracket \delta\Xi\rrbracket_\gamma \diam (w)^\gamma
\end{align*}
Similarly, one estimates the summands in the fourth sum. We conclude by \eqref{eq:core_summation bound for regular methods}.
\end{proof}
\begin{remark}\label{rem:flexibility_remark} 
We leave the following observation as an exercise to the reader: Given a strongly regular method of subdivision $\mcM$, a simplex $\sigma$ and $\mathcal{K}$ a subdivision of $\sigma$. Then there exists a strongly regular method $\mcM' \sim \mcM$ such that $\mcK= \mcM_1(\sigma)$.
\end{remark}
\begin{lemma}
For any strongly regular method of subdivision
$\mathcal{I}_{\mcM} \Xi$ is an element of $\Omega^{k}$.\end{lemma}
\begin{proof}
Note that the first condition of Definition~\ref{def:kform} follows directly, and it remains to check that for any $\sigma \in \mathfrak{X}^{k}$ and subdivision $\mcK | \sigma$, one has  
$$\mathcal{I}_{\mcM}\Xi(\sigma) = \sum_{\sigma' \in \mcK} \mathcal{I}_{\mcM}\Xi(\sigma')\ .$$

The flexibility in Definition~\ref{def:reg_method} allows us to choose a strongly regular method of subdivision $\mcM'$ as in Remark~\ref{rem:flexibility_remark}. 
It follows by Lemma~\ref{lem:M_independent of equivalence class}  that 
\begin{equs}
\mathcal{I}_{\mathcal{M}}\Xi(\sigma)  &= \mathcal{I}_{\mathcal{M}'}\Xi(\sigma) = \lim_{n \rightarrow \infty}  \sum_{\sigma'\in \mcM'_{n+1}(\sigma)} \Xi(\sigma') = \sum_{\sigma'\in \mcM'_{1}(\sigma)}  \lim_{n \rightarrow \infty}  \sum_{\sigma'' \in \mathcal{M}'_{n}(\sigma')} \mathcal{I}_\mathcal{M}' \Xi(\sigma'') \\
&=  \sum_{\sigma'\in \mcK} \mathcal{I}_{\mcM} \Xi(\sigma')\;.
\end{equs}
\end{proof}

%
%\begin{lemma}
%Any map 
%$\mathcal{I}: C^{\eta, \gamma}(\mathfrak{K}) \to \Omega^k(\mathfrak{K}),$
%satisfying
% \[
% |\mathcal{I} \Xi (\sigma)- \Xi (\sigma)|\lesssim \| \delta\Xi\|_{\gamma,\mathfrak{K}} \diam(\sigma)^{\beta}
%\]
%also satisfies
%$$|\mathcal{I}\Xi(\sigma) | \lesssim
%\mass_\eta^k (\sigma) .$$
%\end{lemma}
%\begin{proof}
%
%
%
%Note that in view of $\mass_\eta^k(\sigma) \leq  \diam(\sigma')^{k-1 +\eta}$ it suffices to consider $\Xi\in C^{\eta,\beta}_{\diam}$.
%We fix $h_\sigma>0$ to be the height of $\sigma$. Note that there exists a $\mcK_0$ such that 
%$|\mcK_0|\lesssim_{k} \frac{ \max_{F\in \Bd(\sigma)} \Vol^{k-1}(F)  }{h^{k-1}}$ and $\diam \mcK_0\leq h$.
%Thus we find
%\begin{align*}
%|\mathcal{I}\Xi(\sigma)  | 
%&\lesssim \sum_{\sigma'\in \mathcal{K}}  |\Xi(\sigma') | +  \sum_{\sigma'\in \mathcal{K}}  |\Xi(\sigma')-\mathcal{I}\Xi(\sigma') |
% \diam(\sigma')^{k\bar{\alpha}/\alpha} \\
% &\lesssim \sum_{\sigma'\in \mathcal{K}} \mass_\alpha^k(\sigma')  +  \sum_{\sigma'\in \mathcal{K}} \diam(\sigma')^{\beta}
% \lesssim \sum_{\sigma'\in \mathcal{K}}   \diam(\sigma')^{k-1 +\alpha}\\
% &\lesssim \sum_{\sigma'\in \mathcal{K}}   h_\sigma^{k-1 +\alpha}
% \lesssim \frac{ \max_{F\in \Bd(\sigma)} \Vol^{k-1}(F)  }{h^{k-1}}h_\sigma^{k-1 +\alpha} \lesssim \mass_\alpha^k (\sigma).
%\end{align*}
%\end{proof}

\begin{proof}[Proof of Proposition~\ref{prop:sewing}]
We first note that 
any map 
$\mathcal{I}: C_{2,k}^{\eta, \gamma}(\mathfrak{K}) \to \Omega^k(\mathfrak{K}),$
satisfying
 \[
 |\mathcal{I} \Xi (\sigma)- \Xi (\sigma)|\lesssim \llbracket \delta\Xi\rrbracket_{\gamma,\mathfrak{K}} \diam(\sigma)^{\gamma}
\]
also satisfies
$$|\mathcal{I}\Xi(\sigma)  | 
\leq  |\Xi(\sigma) | +    |\Xi(\sigma)-\mathcal{I}\Xi(\sigma) |\lesssim \llbracket\Xi\rrbracket_{(\eta,\gamma)} \diam(\sigma)^\eta \ .$$
Considering a regular sequence of subdivisions $\{\mcK_n\}_{n\in \mathbb{N}}$ of $\sigma$ any such map also satisfies 
$$
 |\mathcal{I} \Xi (\sigma)- \sum_{\sigma'\in \mcK_n}\Xi (\sigma')|\lesssim \llbracket \delta\Xi\rrbracket_{\gamma,\mathfrak{K}} \sum_{\sigma'\in \mcK_n}  \diam(\sigma')^{\gamma} \to 0\;.
$$

Thus, in view of the preceding lemmas we have constructed for any strongly regular method of subdivision a map $\mathcal{I}$ satisfying the properties in Proposition~\ref{prop:sewing}.
 It remains to prove uniqueness of the map $\mathcal{I}$. Assume $\mathcal{I}'$ is a second such map and let $\mcM$ be a strongly regular method of subdivision. Then, by additivity 
\begin{align*}
|\mathcal{I}\Xi(\sigma) - \mathcal{I}'\Xi(\sigma)  |  &\leq \sum_{\sigma' \in \mcM_n}(\sigma) |\mathcal{I}\Xi(\sigma') - \mathcal{I}'\Xi(\sigma')  |\\
&\leq \sum_{\sigma' \in \mcM_n(\sigma)} \left( |\mathcal{I}\Xi(\sigma') -\Xi(\sigma') | +| \mathcal{I}'\Xi(\sigma')- \Xi(\sigma')  | \right)\\
&\lesssim \sum_{\sigma' \in \mcM_n(\sigma)} \diam(\sigma)^{\gamma} 
\end{align*}
which converges to $0$ as $n\to \infty$ by \eqref{eq:core_summation bound for regular methods}.
\end{proof}
{
\begin{remark}
In the proof of Proposition~\ref{prop:sewing}, we show that the output of the sewing process is independent of any choice of method of subdivision, while 
Lemma~\ref{lem:M_independent of equivalence class} only showed independence within each equivalence class, where these equivalence classes depend in an ad hoc way on how we choose to quantify properties of methods of subdivision in Definition~\ref{def:reg_method}.
\end{remark}
}

%Then 
%$$S_\mcM (A)(\sigma)=\lim_{n\to \infty} \sum_{\sigma' \in \mathcal{M}_1 (\sigma)} A(\sigma')$$ exists
%and defines an element of $\Omega^k$ which satisfies $|S_\mcM (A)(\sigma)\lesssim \Vol^{k}(\sigma)^\eta$
%as well as 
%for any decomposition $\mathcal{K}$ of $\sigma$
%$$|S_\mcM A(\sigma)-\sum_{\sigma' \in\mathcal{K} } A(\sigma')|\lesssim_{\mcM} \sum_{\sigma' \in \mathcal{K} } \diam(\sigma)^{k+\eps}\ .$$
%}

\begin{appendix}
\section{Symbolic Index}
\begin{center}
\begin{longtable}{p{.12\textwidth}p{.65\textwidth}p{.12\textwidth}}
\toprule
Object & Meaning & Ref. \\
\midrule
\endhead
\bottomrule
\endfoot
$\Vol^k$ & $k$-dimensional Hausdorff measure & Page~\pageref{sec:preliminaries}\\
$\pmb{\sigma}$& non-oriented simplex & Def.~\ref{def:non-oriented}\\
$\sigma$& oriented simplex & Def.~\ref{def:simplex, etc} \\
$\VV(\sigma)$ & vertices of $\sigma$ &  Def.~\ref{def:non-oriented} \\
$\diam$& diameter of a simplex or cube and mesh of a subdivision  & Def.~\ref{def:subdivision,mesh}  \\
$\mathcal{K}|\sigma$& $\mathcal{K}$ is a subdivision of $\sigma$ & Def.~\ref{def:subdivision_2} \\
%$\mathfrak{K}$ & compact subset of $\mathbb{R}^d$ &\\
$\mathfrak{X}^k$  & set of oriented simplices & Def.~\ref{def:simplex, etc} \\
$\mathcal{X}^k $ & $\mathbb{Z}$ module of chains & Def.~\ref{def:chain}\\
${\mass^k_\alpha }$ & $\alpha$-mass & Eq.~\ref{def:mass} \\
$\Omega^k$ & universe of (measurable) cochains & Def.~\ref{def:kform}  \\
$\Omega_{\alpha,\beta}^{k}$& elements $A\in \Omega^k$ such that $\|A\|_{\alpha,\beta}<+\infty$  & Def.~\ref{def:dist_form}  \\
$\mathfrak{Q}^k$  & set of oriented cubes & Def.~\ref{def:cube} \\
$\Bd(\sigma)$ & oriented faces of $\sigma$ & Def.~\ref{def:simplex, etc}  \\
$\partial$ & boundary operator $\delta \sigma = \sum_{F\in \Bd(\sigma)} F$ & Eq.~\ref{eq:boundary op}  \\
$\mathcal{B}_{\alpha,\beta}$& completion of $(\mathcal{X}^k,  |\cdot |_{(\alpha,\beta)}) $ & Def.~\ref{def:B^k} \\
$C_2^{\eta,\gamma}$ & certain functions on simplices/germs & Def.~\ref{def:germ}  \\
$\Xi$& element of $C_2^{\eta,\gamma}$  & Def.~\ref{def:germ} \\
$\mu_{\sigma}$& family of probability measure supported on $\pmb{\sigma}$ & Page~\pageref{thm:multiplication} \\
$\delta\Xi$& generalised increment of $\Xi$ & Def.~\ref{def:germ}\\
$\mathcal{I}$ & sewing/reconstruction operator $C_2^{\eta,\gamma}\to \Omega^k$ & Prop.~\ref{prop:sewing}\\
$|\cdot |_{(\alpha,\beta)} $ & $(\alpha,\beta)$-flat norm & Def.~\ref{def:flat_norm}\\
$\delta_{\pmb Q}$& element of $D'$ which integrates functions over $\pmb{Q}$ & Def.~\ref{def:cube} \\
$\mathcal{M}$& method of subdivision & Def.~\ref{def:reg_method} \\
$\mathfrak{e}$& eccentricity of a simplex $\mathfrak{e} (\sigma)= \frac{\diam(\sigma)^k}{\Vol^k(\sigma)}$ & Eq.~\ref{eq:eccentricity} \\
\end{longtable}
\end{center}

\end{appendix}

	\bibliographystyle{Martin}
	\bibliography{./GeoInt.bib}

\end{document}